\documentclass[a4paper,10pt]{amsart}

\usepackage[utf8]{inputenc}
\usepackage[T1]{fontenc}
\usepackage{amsfonts}
\usepackage{amssymb}
\usepackage{enumerate}
\usepackage{amsthm}
\usepackage{diagrams}
\usepackage{epic}
\usepackage{curves}
\usepackage[dvips]{rotating}

\newcommand{\Stratum}{\mathcal{H}}      
\newcommand{\C}{\mathbb{C}}             
\renewcommand{\H}{\mathbb{H}}           
\renewcommand{\P}{\mathbb{P}}           
\newcommand{\R}{\mathbb{R}}             
\newcommand{\Z}{\mathbb{Z}}             
\newcommand{\SO}{\mathrm{SO}}           
\newcommand{\SL}{\mathrm{SL}}           
\newcommand{\GL}{\mathrm{GL}}           
\newcommand{\Aut}{\mathrm{Aut}}         
\newcommand{\Mod}{\mathrm{Mod}}         
\newcommand{\Aff}{\mathrm{Aff}}         
\newcommand{\Trans}{\mathrm{Trans}}     
\newcommand{\Stab}{\mathrm{Stab}}       
\newcommand{\FP}{\mathrm{FP}}           
\newcommand{\TM}{\mathcal{T}}           

\newcommand{\Fund}{\mathcal{F}}         
\newcommand{\Curve}{\mathcal{C}}        
\newcommand{\Surface}{\mathcal{S}}      
\newcommand{\der}{D}         		
\newcommand{\pr}{\mathrm{pr}}           
\newcommand{\barpr}{\mathrm{\overline{pr}}} 
\newcommand{\id}{\mathrm{id}}           
\newcommand{\ord}{\mathrm{ord}}         
\newcommand{\isom}{\cong}               
\renewcommand{\bar}{\overline}
\newcommand{\eps}{\varepsilon}

\newcommand{\txt}[1]{\mathrm{#1}}              
\newcommand{\ra}{\rightarrow}           

\newcommand{\set}[2]{\left\{#1\mid #2\right\}} 

\newcommand{\gen}[1]{\langle#1\rangle}

\newcommand{\beq}{\begin{eqnarray*}}     
\newcommand{\eeq}{\end{eqnarray*}}       
\newcommand{\beqn}{\begin{eqnarray}}     
\newcommand{\eeqn}{\end{eqnarray}}       

\newcommand{\tcircle}{\makebox(10,7)[c]{\makebox[1ex][l]{\circle{7}}}}              
\newcommand{\tcircleblack}{\makebox(10,7)[c]{\makebox[1ex][l]{\circle*{8}}}} 

\newcommand{\textmatrix}[4]{\left(\begin{smallmatrix} #1&#2\\ #3&#4\\ \end{smallmatrix}\right)}
\newcommand{\textvector}[2]{\left(\begin{smallmatrix} #1\\ #2\\ \end{smallmatrix}\right)}

\newcommand{\eg}{e.\,g.\ }		
\newcommand{\ie}{i.\,e.\ }		
\newcommand{\wrt}{w.\,r.\,t.\ }		
\newcommand{\define}[1]{\emph{#1}}

\theoremstyle{plain}
\newtheorem{thm}{Theorem}
\newtheorem{prop}{Proposition}[section]
\newtheorem{lem}[prop]{Lemma}
\newtheorem{cor}[prop]{Corollary}

\theoremstyle{definition}
\newtheorem{defn}[prop]{Definition}

\theoremstyle{remark}
\newtheorem{rem}[prop]{Remark}

\numberwithin{equation}{section}

\title{An origami of genus 2 with a translation}
\author{F. Herrlich}
\author{A. Kappes}
\author{G. Schmith\"usen}
\date{\today}

\thanks{This paper partially draws on results from an ongoing research project 
comissioned by the LANDESSTIFTUNG Baden-Württemberg}
\begin{document}

\begin{abstract}
We study an example of a Teichmüller curve $\Curve_S$ in the moduli space $M_2$ coming from an origami $S$. It is particular in that its points admit $V_4$ as a subgroup of the automorphism group. We give an explicit description of its points in terms of affine plane curves, we show that $\Curve_S$ is a nonsingular, affine curve of genus 0 and we determine the number of cusps in the boundary of $M_2$.
\end{abstract}

\maketitle


\section*{Introduction}
An origami is a compact surface $X$ of genus $g$ that arises from gluing finitely many Euclidean unit squares along their edges. If one uses only translations to identify edges, one obtains in a natural way a translation surface. Affine deformations of the translation structure yield new translation structures on $X$, and in particular a variation of the complex structure. One gets a geodesic disc in the Teichmüller space of compact Riemann surfaces of genus $g$, which in the case of origamis always projects to an algebraic curve, called an origami curve, in the moduli space $M_g$ of compact Riemann surfaces of genus $g$. This provides a means of studying the geometry of the moduli space by looking at complex curves in it.


This paper is devoted to the study of a particular origami $S$ of genus 2 and its curve $\Curve_S$ in the moduli space $M_2$. A picture of $S$ is drawn in Figure \ref{Figure: Origami S}; edges with the same letter are identified and $\square$, $\blacksquare$, $\tcircle$ and $\tcircleblack$ are the four vertices of the square-tiling.
\begin{figure}[ht]
\setlength{\unitlength}{1cm}
\begin{center}
\begin{picture}(4,2)
\put(0,0){\framebox(1,1){1}}
\put(1,0){\framebox(1,1){2}}
\put(2,0){\framebox(1,1){3}}
\put(1,1){\framebox(1,1){4}}
\put(2,1){\framebox(1,1){5}}
\put(3,1){\framebox(1,1){6}}

\put(0.5,-.3){\scriptsize{$c$}}
\put(0.5,1.1){\scriptsize{$c$}}
\put(1.5,-.3){\scriptsize{$d$}}
\put(1.5,2.1){\scriptsize{$d$}}
\put(2.5,-.3){\scriptsize{$e$}}
\put(2.5,2.1){\scriptsize{$e$}}
\put(3.5,0.66){\scriptsize{$f$}}
\put(3.5,2.1){\scriptsize{$f$}}

\put(-.25,.45){\scriptsize{$a$}}
\put(3.1,.45){\scriptsize{$a$}}
\put(0.75,1.45){\scriptsize{$b$}}
\put(4.1,1.45){\scriptsize{$b$}}

\put(-.15,-.15){$\square$}
\put(-.15,.88){$\square$}
\put(2.88,.88){$\square$}
\put(2.88,-.15){$\square$}
\put(2.88,1.88){$\square$}

\put(.88,-.15){$\blacksquare$}
\put(.88,.88){$\blacksquare$}
\put(.88,1.88){$\blacksquare$}
\put(3.88,.88){$\blacksquare$}
\put(3.88,1.88){$\blacksquare$}

\put(2,0){\circle{0.2}}
\put(2,2){\circle{0.2}}
\put(2,1){\circle*{0.2}}

\end{picture}\\[2mm]
\end{center}
\caption{The origami $S$\label{Figure: Origami S}}
\end{figure}
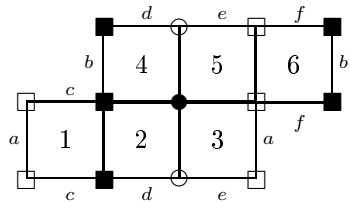

Besides the hyperelliptic involution, $S$ admits a translation $\tau$ of order 2, \ie a permutation of the squares that respects the gluings, namely $(1\, 6)(2\, 4)(3\, 5)$. Therefore, $\Curve_S$ lies in the subvariety of $M_2$, whose points admit a subgroup isomorphic to the Klein four group $V_4$ as automorphism group. This property will enable us to give a concrete description of the points on $\Curve_S$ in terms of affine plane curves.

\begin{thm}\label{Introduction - Thm - Main Result}
The origami curve $\Curve_S$ of the origami $S$ is equal to the projection of the affine curve $V\subset \C^2$ to the moduli space $M_2$, where $V$ is given by
\[V:\quad \mu(\lambda +1) - \lambda = 0,\quad \lambda \neq 0, \pm 1, -\tfrac{1}{2}, -2.\]
Explicitly, every point on $\Curve_S$ is birational as affine plane curve to
$$y^2 = (x^2-1)(x^2-\lambda^2)(x^2-\bigl(\tfrac{\lambda}{\lambda+1}\bigr)^2)$$
for some $\lambda \in \C\setminus \{0,\pm 1, -\tfrac{1}{2},-2\}$. Moreover, the curve $\Curve_S$ is an affine, regular curve of genus 0 with 2 cusps.
\end{thm}

The proof uses the fact that the points $X\in M_2$ with $\Aut(X)\supseteq V_4$ can be described by affine plane curves, whose equations depend on two complex parameters $(\lambda,\mu)$. The main observation is that certain points on the origami become 3-torsion points on an elliptic curve, namely on the quotient surface of $S$ by $\gen{\tau}$; simultaneously, we know the coordinates of these points on the affine plane curve (depending on $\lambda$ and $\mu$). From this we get a relation between $\lambda$ and $\mu$, which gives the equation of the origami curve $\Curve_S$.

This is a fascinating interplay between objects that are  defined by analytical means (translation surfaces and geodesic discs in the Teichmüller space) and algebraic objects (algebraic curves in $M_g$), and it is a priori not at all obvious how one can find links between these worlds. However, our origami $S$ is by no means special among the origamis in genus 2 that have a translation. In fact, our construction generalizes to this class of origamis, which will be discussed in a future paper.

This paper grew out of the diploma thesis of one of the authors, André Kappes \cite{Ka07}.

\subsection*{Structure of this paper}
In the first section, we fix some notations and recall briefly how and why an origami defines a curve in the moduli space.

The second section is devoted to a description of the moduli space $M_2$ in terms of affine plane curves and to a discussion of loci with many automorphisms. We follow a description of Geyer \cite{Gey74}. This section is fundamental for the following discussion of the origami curve $\Curve_S$ and the proof of Theorem \ref{Introduction - Thm - Main Result}.

In the third section we discuss how the group of common automorphisms of points on a Teichmüller curve can look like in genus 2.

The fourth section supplies a proof of Theorem \ref{Introduction - Thm - Main Result}.

\subsection*{Related work}
The term \define{origamis} originated with \cite{Lo03}, but they have also been investigated by other authors under the name \define{square-tiled surfaces}. They belong to the more general class of flat surfaces, which have been studied extensively during the last years in algebraic geometry, complex analysis and dynamical systems. In this paper, we restrict to what one sometimes calls oriented origamis: they give rise to translation surfaces.

Only for a few origami curves, the algebraic equations of their points are known. Möller \cite{Moe05} gives equations for two example origamis in genus 2. In \cite{LS07}, we are given equations for all origamis of genus 2 that are tiled by 4 squares. The authors also present different families of hyperelliptic square-tiled surfaces parametrized by the genus and exhibit their equations. The $\SL_2(\R)$-orbits of square-tiled surfaces in the stratum $\Stratum(2)$ where classified by \cite{HL06} and \cite{McM05}; to our knowledge, this is still open for the stratum $\Stratum(1,1)$, to which the Teichmüller disk of the origami $S$ belongs. 

In genus 3, there is a particularly interesting origami with many nice properties \cite{HS05}; its curve intersects infinitely many other origami curves. These curves and the corresponding origamis are investigated in \cite{HS07}.

In \cite{Her06}, we are given equations of an infinte family of origamis (whose members have arbitrarily high genus). They are called Heisenberg origamis and belong to the class of characteristic origamis, \ie their Veech group is the entire group $\SL_2(\Z)$.

Origamis are somewhat more accessible than general translation surfaces, for one can compute their Veech groups explicitly \cite{Sc04} and one knows that they always define an algebraic curve in the moduli space.

\section{Translation surfaces, Teichmüller disks and origamis}
In this section, we give a short review of the general theory of Teichmüller disks and curves and origamis and origami curves in particular. References for this part are \eg \cite{Ve89}, \cite{GJ00}, \cite{EG97}, \cite{McM03}, \cite{Sc05}, \cite{HS06} to list only some.

\subsection{Notations}
We first fix some notations. If $X$ is a Riemann surface, we write $\Aut(X)$ for the group of holomorphic automorphisms of $X$. If $\theta\in \Aut(X)$, then $\FP(\theta)$ denotes the set of fixed points. In the following $I$ always denotes the identity matrix (of the appropriate dimension).

\subsection{Translation surfaces}\label{Section 1 - Translation surfaces}
Let $\omega$ be a nonzero holomorphic differential on a compact Riemann surface $X$. We can define an atlas on $X \setminus Z(\omega)$, where $Z(\omega)$ is the set of zeros of $\omega$, by using local primitives of $\omega$ as charts. We get a \define{translation surface} $(X,\omega)$, \ie the transition maps between two charts are locally translations of $\C$.

A point $P\in Z(\omega)$ leads to a singularity of the translation structure: It is a conical point with a cone angle of $2\pi(d+1)$, where $d$ is the multiplicity of the zero. By Riemann-Roch, $\omega$ has precisely $2g-2$ zeros counted with multiplicities. The moduli space $\Omega M_g$ of pairs (compact Riemann surface, holomorphic differential) is stratified by the multiplicities. In particular, $\Omega M_2$ consists of two strata $\Stratum(1,1)$ and $\Stratum(2)$, which correspond to holomorphic differentials with two simple zeros, resp. one double zero. 

On $X\setminus Z(\omega)$, one can define a flat Riemannian metric by pulling back the Euclidean metric via the coordinate charts. Geodesics for that metric are straight line segments; geodesics that connect two singularities are called \textit{saddle connections}. The \textit{lattice of relative periods} is the subgroup of $\R^2$ spanned by the vectors corresponding to saddle connections.

We always identify $\C$ with $\R^2$ by sending $\{1,i\}$ to the standard basis. Then translations are biholomorphic, and we can also view the translation structure as a complex structure on $X\setminus Z(\omega)$: We have a Riemann surface of finite type and the associated compact surface is again $X$.

\subsection{Affine diffeomorphisms}
Given translation surfaces $(X,\omega)$, $(Y,\nu)$ as above, we say that a diffeomorphism $f:X\ra Y$ is \define{affine} (\wrt the respective translation structures), if, in local coordinates, $f$ is given by
\[z\mapsto A\cdot z + t\]
for some $A\in \GL_2(\R)$ and $t\in \R^2$. If $f$ is affine, then its matrix part $A$ is globally the same. If $f$ is orientation preserving, then $A\in \GL_2^+(\R)$, and since $X$ is of finite volume, we have $A\in \SL_2(\R)$. The affine orientation preserving diffeomorphisms $X\ra X$ form a group $\Aff^+(X,\omega)$.

We get a map $\der: \Aff^+(X,\omega) \ra \SL_2(\R)$ by assigning to $f$ its matrix part $A$. The image of $\der$ is called the \define{Veech group} of $(X,\omega)$, and is denoted it by $\Gamma(X,\omega)$. Its kernel is the group of translations $\Trans(X,\omega)$, \ie maps that are automorphisms for the translation structure. Finally, we call $\Aut(X,\omega) = \der^{-1}(\{\pm I\})$ the group of biholomorphic automorphisms of $(X,\omega)$. 


\subsection{Teichmüller spaces and moduli spaces}
Let $X$ be a compact Riemann surface of genus $g$. We denote by $\TM(X)$ the Teichmüller space with base point $X$, \ie points in $\TM(X)$ are isomorphism classes of marked Riemann surfaces $(S,f)$, where $S$ is a Riemann surface and $f:X\ra S$ is an orientation preserving diffeomorphism (called a Teichmüller marking). As $\TM(X)$ only depends on the topological type of the chosen reference surface $X$, we also write $\TM(X) = \TM_{g}$.
The (coarse) moduli space of Riemann surfaces of genus $g$ is denoted by $M_g$.

The mapping class group of $X$ is the group of isotopy classes of orientation preserving diffeomorphisms, and we denote it by $\Mod(X) = \Mod(g)$. 
The mapping class group acts on the Teichmüller space and the quotient of this action is the corresponding moduli space. More precisely, if $P = (S,f)\in \TM(X)$ and $a$ is an orientation preserving diffeomorphism of $X$ then $a\cdot P = (S, f \circ a^{-1})$.

\subsection{Teichmüller disks}
A translation surface $(X,\omega)$ defines a Teichmüller disk in $\TM_{g} = \TM(X)$ in the following way. Let $A\in \SL_2(\R)$ and let us denote by $A\cdot (X,\omega)$ the translation surface obtained by postcomposing each chart with the $\R$-linear map $z\mapsto A\cdot z$. We consider the identity map on $X$ as an affine orientation preserving diffeomorphism $f_A:(X,\omega)\ra A\cdot(X,\omega)$ \wrt to the translation structures. In this way, $A$ naturally defines a point $P_A = (A\cdot (X,\omega), f_A)$ in $\TM(X)$. Thus we get a map
\[\SL_2(\R) \ra \TM(X),\ A \mapsto P_A.\]
As the action of a matrix in $\SO_2(\R)$ stabilizes $(X,\omega)$, this map factors as
\[\SO_2(\R)\backslash\SL_2(\R) = \H \rTo^{j} \TM(X),\quad \SO_2(\R)\cdot A \mapsto P_A.\]
By an appropriate choice of the identification of the upper half plane with the set $\SO_2(\R)\backslash \SL_2(\R)$, the map $j$ is a holomorphic embedding, which is an isometry for the Poincar\'e metric on $\H$ and the Teichmüller metric on $\TM(X)$. Its image is the Teichmüller disk $\Delta(X,\omega^2)$ associated to $(X,\omega)$: It is a complex geodesic in $\TM(X)$, which corresponds to the base point $X$ and the cotangent vector $\omega^2 \in (T_X\TM(X))^*$. Note that the cotangent space can be identified with the space of holomorphic quadratic differentials on $X$.


Given a Teichmüller embedding $j:\H\ra \TM(X)$ that arises from a translation surface $(X,\omega)$, we consider the stabilizer $\Stab(j(\H)) \subset \Mod(X)$. By \cite[Lemma 5.2, Theorem 1]{EG97}, this group is isomorphic to $\Aff^+(X,\omega)$. If $P_A\in j(\H)$ and $f\in \Aff^+(X,\omega)$ with $B = \der(f)$, then $f\cdot P_A = P_{AB^{-1}}$, thus $f$ stabilizes $j(\H)$. Now the action of $\Stab(j(\H))$ might not be effective on $\Delta(X,\omega^2)$, \ie there might be a nontrivial pointwise stabilizer $\Stab_0(j(\H)) \subset \Stab(j(\H))$.

\begin{prop}\label{Section 1 - Proposition: common automorphisms of TM-disk are affine}
We have $\Stab_0(j(\H)) \isom \Aut(X,\omega) = \der^{-1}(\{\pm I\})$, and this is the group of common automorphisms of the points in the Teichmüller disk $\Delta(X,\omega^2)$.
\end{prop}
\begin{proof}
Note that $\theta \in \Stab_0(j(\H))$ can be considered as an element of $\Aut(X)$. It follows from the definition of the action of $\Mod(g)$ that $\theta$ stabilizes each point in $\Delta(X,\omega^2)$, if and only if $f_A\circ \theta \circ f_A^{-1}$ is a biholomorphic automorphism of $A\cdot (X,\omega)$ for all $A\in \SL_2(\R)$. We describe $f_A\circ\theta\circ f_A^{-1}$ in local coordinates. Thus we can break down the argument to the case that we are given a holomorphic map $h:U\ra \C$ on a domain $U\subset \C$, for which the map
\[(z\mapsto A\cdot z) \circ h\circ (z\mapsto A^{-1}\cdot z)\]
is holomorphic for all $A\in \SL_2(\R)$.
Let $z_0\in U$ and let $B = B(z_0)$ be the real derivative of $h$ at $z_0$. By the Cauchy-Riemann differential equations, $B\in \R^\times\cdot\SO_2(\R)$. By our assumption, $ABA^{-1}\in \R^\times\cdot\SO_2(\R)$. If we plug in $A = \left(\begin{smallmatrix}
					1 & 1 \\
					0 & 1
                                      \end{smallmatrix}\right)$, we find that $B\in \R^\times\cdot I$, which implies $h'(z_0) \in \R^\times$. So $h'(U) \subset \R$, which is not open in $\C$. Since $h'$ is holomorphic on $U$, this forces $h'$ to be constant. Therefore, $\theta$ is affine with derivative $B\in \R^\times\cdot I \cap \{\pm I\}$, \ie $\theta \in \Aut(X,\omega)$.
On the other hand, every element in $\Aff^+(X,\omega)$ with derivative $\pm I$ stabilizes each point in $\Delta(X,\omega^2)$.
\end{proof}

Now that we have defined Teichmüller disks (which are a priori analytical objects), let us study under which conditions they give rise to algebraic curves in the moduli space. Let $\hat\Gamma(X,\omega) = R\Gamma(X,\omega)R^{-1}$ be the mirror Veech group (where $R = \left(\begin{smallmatrix}
	-1 & 0 \\
	 0 & 1 \\
      \end{smallmatrix}\right)$). Then the map $j$ is equivariant for the action of $\hat\Gamma(X,\omega)$ on $\H$ (via Möbius transformations) and $\Aff^+(X,\omega)$ on $j(\H)$. (Note that the two groups are linked via $\Aff^+(X,\omega) \ra \hat\Gamma(X,\omega)$, $f\mapsto R\der(f)R^{-1}$.) Therefore, we can pass on both sides to the quotient.
A \define{Teichmüller curve} is an algebraic curve in the moduli space that is the image of a Teichmüller disk (under the natural projection). The following criterion determines, when Teichmüller curves occur.

\begin{prop}[see {\cite[Corollary 3.3]{McM03}}]\label{Section 1 - Proposition: Teichmueller Curve iff Veech group is lattice}
A Teichmüller disk $\Delta(X,\omega^2)$ projects to a Teichmüller curve $\Curve$ in the moduli space, if and only if the Veech group $\Gamma(X,\omega)$ is a lattice in $\SL_2(\R)$, \ie $\H/\Gamma(X,\omega)$ is a Riemann surface of finite type. In this case, $\H/\hat\Gamma(X,\omega)$ is the normalization of $\Curve$.
\end{prop}

\subsection{Origamis}
An origami can be depicted as follows. Take a finite number of unit squares in the plane and glue the upper edge of a square to a lower edge of a square and the left edge of a square to a right edge of a square. We require the edges to be identified by a translation, and the resulting surface to be connected. This yields a tiling of our surface into squares, and we have a covering map to the origami that consists of only one square, i.e. a torus. More precisely, we make the following definition.

\begin{defn}
Let $E$ be a topological torus and $\bar P\in E$. An \define{origami} is a (finite) covering $O = (p:X\ra E)$, where $X$ is a compact topological surface of genus $g\geq 1$, such that $p$ is ramified at most over the point $\bar P$.
\end{defn}

Note that an origami is a topological or even combinatorial object; the topological structure is given precisely by a monodromy representation $\pi_1(E) \isom F_2 \ra S_d$, where $d$ is the degree of the covering $p:X\ra E$.

Let $O= (p:X\ra E)$ be an origami. We choose a complex structure on $E$, i.e. we take the torus $E_A$, which is defined as follows. Let $A\in \SL_2(\R)$, and let $\Lambda_A$ be the lattice in $\C$ spanned by the columns of $A$. Then $E_A = \C/\Lambda_A$ and $\bar P = 0+\Lambda_A$. We can pull the holomorphic differential $\omega_A$ on $E_A$ (which is unique up to a scalar in $\C$) back to a holomorphic differential $p^*\omega_A$ and end up with a translation surface $X_A = (X,p^*\omega_A)$. Every point on the Teichmüller disk $\Delta(X,(p^*\omega_A)^2)$ then arises by running through all possible complex structures that can be put on $E$, i.e. by running through all of $\SL_2(\R)$. In particular, we have $B\cdot X_A = X_{BA}$.



The Teichmüller disk $\Delta_O = \Delta(X,(p^*\omega_A)^2)$ of an origami $O = (p:X\ra E)$ only depends on the combinatorial data of the covering $p$ and not on the choice of the base point $X_A$. In particular, we can work with the base point $X_I$, which is a cover of the standard torus $E_I$, whence $X_I$ is called \textit{square-tiled}. Furthermore, the Veech groups of points in the same Teichmüller disk are all conjugated in $\SL_2(\R)$. We define the \textit{Veech group of the origami} $O$ to be the group $\Gamma(O) := \Gamma(X,p^*\omega_I)$.



An origami always defines a Teichmüller curve in the moduli space, for the following theorem implies that $\Gamma(O)$ always is a lattice in $\SL_2(\R)$. The corresponding Teichmüller curve to an origami $O$ is called \textit{origami curve} $\Curve_O$.

\begin{prop}[see {\cite[Theorem 5.5]{GJ00}}] \label{Section 1 - Gutkin and Judge say...}
A translation surface $(X,\omega)$ is square-tiled, if and only if the groups $\SL_2(\Z)$ and $\Gamma(X,\omega)$ are commensurate.
\end{prop}

Note that the Veech group of an origami need not be a subgroup of $\SL_2(\Z)$ (see \eg \cite{Moe05} for an example). This is the case, if one uses too many squares in the tiling. (One could for instance construct a new origami out of a given one by subdividing each square into 4 smaller squares.) One way to circumvent this problem is to consider $X^* = X\setminus p^{-1}(\bar P)$ instead of $X$, which forces every affine diffeomorphism to descend to the torus $E_I$, whence to have matrix part in $\SL_2(\Z) = \Gamma(E_I,\omega_I)$. Then the Veech group $\Gamma(X^*,p^*\omega_I) = \Gamma(X,p^*\omega_I) \cap \SL_2(\Z)$ has finite index in $\Gamma(X,p^*\omega_I)$ by Proposition \ref{Section 1 - Gutkin and Judge say...}.

A more convenient way is to consider primitive origamis: an origami is called \textit{primitive}, if the lattice of relative periods $\Lambda(\omega_I)$ is $\Z^2$.

\begin{rem}[see \eg {\cite[Lemma 2.3]{HL06}}] \label{Section 1 - Remark: Primitive Origamis have entire Veech group}
If $O$ is a primitive origami, then $\Gamma(O) \subset \SL_2(\Z)$.
\end{rem}

Note that the origami $S$ is primitive, since the vectors $\textvector{1}{0}$ and $\textvector{0}{1}$ are contained in the lattice of relative periods.

\section{Points in genus 2 with more automorphisms}
Let $M_2$ be the moduli space of compact Riemann surfaces of genus 2. As we will see later in Section 4, every point $X$ of our origami curve $\Curve_S$ has a subgroup $G$ of $\Aut(X)$ that is isomorphic to $V_4$, the Klein four group. $G$ is generated by the translation $\tau$ of $S$ and the hyperelliptic involution $\sigma$.

In this section, we describe the points on the subvariety of $M_2$ that admit $V_4$ as a subgroup of their automorphism group. This can be found \eg in \cite{Gey74} or \cite{Ig60}, where points of $M_2$ are classified according to their automorphism group.

\subsection{Automorphism groups of points in $M_2$}
In general, a compact Riemann surface $Y$ of genus 2 carries a (unique) hyperelliptic involution $\sigma \in \Aut(Y)$. Let $\phi : Y\ra Y/\gen{\sigma} \isom \P^1$ denote a quotient map. Then $\phi$ is ramified precisely over a six-point set $B\subset \P^1$, which is the image of the set $\FP(\sigma)$. Note that $\phi$ is unique up to composition with an element of $\Aut(\P^1)$. Since $\sigma$ is central, every automorphism of $Y$ descends to $\P^1$ and induces a permutation of $B$. Conversely, every $\theta\in \Aut(\P^1)$ that satisfies $\theta(B) = B$ can be lifted to $Y$. So instead of studying $\Aut(Y)$, we can look at $\bar{\Aut(Y)} = \Aut(Y)/\gen{\sigma} \subset \Aut(\P^1)$. Note that $\bar{\Aut(Y)}$ is isomorphic to a subgroup of the symmetric group $S(B) = S_6$; every $\theta\in \bar{\Aut(Y)}$ has two fixed points and every orbit that does not contain a fixed point has the same number of elements.

In the following, let $C_n$ denote the cyclic group of order $n$ and let $D_n$ denote the dihedral group of order $2n$. For an overall view of $M_2$ we cite the following proposition.

\begin{prop}[see {\cite[Satz 3, Satz 4]{Gey74}}]\label{Section 2 - Propositon: Overview of M_2}
$M_2$ is a 3-dimensional, rational, normal, affine variety with one singular point $P$ which corresponds to the unique $Y\in M_2$ with $\bar{\Aut(Y)} \isom C_5$.

The points $Y\in M_2$ with $\bar{\Aut(Y)}\supseteq C_2$ form a rational surface $\Surface \subset M_2$. For a generic point $Y\in \Surface$, one has $\bar{\Aut(Y)} \isom C_2$. There are two rational curves $U$ and $U' \subset \Surface$, where $\bar{\Aut(Y)}$ is bigger. For each point $Y$ on $U$, we have $\bar{\Aut(Y)} \supseteq V_4$ and for each point $Y\in U'$, we have $\bar{\Aut(Y)}\supseteq S_3$. In both cases, we have equality, except for two points $Q$ and $Q'$ where $U$ and $U'$ intersect. There, $\bar{\Aut(Q)} \isom S_4$ and $\bar{\Aut(Q')} \isom D_6$.

The surface $\Surface$ is not normal, and its singular locus is precisely the curve $U$.
\end{prop}

\subsection{Points in the surface $\Surface$}
We will develop a more precise description of $\Surface$. Given a point $Y\in \Surface$, let $\Aut(Y)$ have the subgroup $\{\id,\sigma,\tau,\sigma\tau\}\isom V_4$, where $\sigma$ is the hyperelliptic involution on $Y$ and $\tau$ is another involution. First, we show that we can choose $\tau$ such that $\FP(\tau)\cap \FP(\sigma) =\emptyset$.

\begin{lem}\label{Section 2 - Lemma: wlog sigma and tau have no common fixed point}
If $\FP(\tau) \cap \FP(\sigma) \neq \emptyset$, then there is an involution $\tau'\in \Aut(Y)$ such that $\FP(\tau')\cap \FP(\sigma) = \emptyset$.
\end{lem}
\begin{proof}
Since $\sigma$ is central, the map $\tau$ descends to $\bar\tau \in \Aut(\P^1)$. Moreover, $\sigma$ permutes the fixed points of $\tau$; it follows from the Riemann-Hurwitz formula, that this is a 2-element set, so by our assumption, $\sigma$ fixes $\FP(\tau)$ pointwise. Thus, $\phi(\FP(\tau)) \subset B$; we postcompose $\phi:Y\ra \P^1$ with a Möbius transformation that sends $\phi(\FP(\tau))$ to $\{0,\infty\}$ and a third element $b\in B$ to $1$. In this way, $\bar\tau$ becomes the map $z\mapsto -z$ and
\[B = \{0,\infty, 1, -1, \alpha, -\alpha\}\]
for some $\alpha \in \C\setminus \{0,\pm 1\}$. The set $B$ is also invariant, if we apply
$z\mapsto \alpha z^{-1}$ (as well as $z\mapsto -\alpha z^{-1}$), which is an involution without fixed points in $B$, hence lifts to an involution $\tau' \in \Aut(Y)$ as desired.
\end{proof}

Note that if $\tau$ and $\sigma$ have no common fixed point, then $\FP(\tau)$, $\FP(\sigma)$ and $\FP(\sigma\tau)$ are mutually disjoint. Proceeding as in the above proof, we can show the first part of the following lemma.

We define the parameter space $P$ to be
\beqn\label{Section 2 - Equation: Parameter space} P = (\C\setminus\{0,\pm 1\})^2 \setminus (\Delta \cup \Delta'),\eeqn
where $\Delta \subset \C^2$ denotes the diagonal and $\Delta' = \set{(z,-z)}{z\in \C}$.

\begin{lem}\label{Section 2 - Lemma: Choice of covering map and parameters}
\begin{enumerate}[a)]
\item Let $Y\in \Surface$ and let $\tau\in \Aut(Y)$ be a fixed involution such that $\FP(\tau)\cap \FP(\sigma) =\emptyset$. Then there exists a quotient map $\phi:Y\ra \P^1$ for the action of the hyperelliptic involution $\sigma$ on $Y$, such that the automorphism $\tau\in\Aut(Y)$ descends to the map
\[\bar\tau:\P^1\ra\P^1,\ z\mapsto -z,\]
and such that the set of branch points of $\phi$ is of the form
\[B = \{1,-1,\lambda,-\lambda,\mu,-\mu\},\]
where $(\lambda,\mu) \in P$.
\item If $\phi':Y\ra \P^1$ is a map with the same properties as $\phi$, then $\phi' = \delta\phi$, where $\delta:\P^1\ra\P^1$ is one of the maps in the set
\begin{equation}\label{Section 2 - Equation: List of deltas}
\begin{split}
  &\{\,\id,\ (z\mapsto \lambda^{-1}z),\ (z\mapsto \mu^{-1}z)\,\} \\
  \cup& \   \{\,\bar\tau,\ (z\mapsto \lambda^{-1}z)\circ\bar\tau,\ (z\mapsto \mu^{-1}z)\circ \bar\tau\,\}\\
  \cup& \  \{\,(z\mapsto z^{-1}),\ (z\mapsto \lambda z^{-1}),\ (z\mapsto \mu z^{-1})\,\} \\
  \cup& \  \{\,(z\mapsto z^{-1})\circ \bar\tau,\ (z\mapsto \lambda z^{-1})\circ\bar\tau,\ 
                   (z\mapsto \mu z^{-1})\circ \bar\tau\,\}.
\end{split}
\end{equation}
\end{enumerate}
\end{lem}
\begin{proof}
Only part b) needs to be justified. By the general theory, there exists $\delta\in\Aut(\P^1)$, such that $\phi' = \delta\phi$. The map $\phi'$ satisfies $\bar\tau\phi' = \phi'\tau$. Thus we have $\bar\tau\delta\phi = \delta\bar\tau\phi,$ which leads to $\bar\tau = \delta\bar\tau\delta^{-1}$, because $\phi$ is surjective. Then $\delta$ permutes the fixed points of $\bar\tau$. So either $\delta(0) = \infty$ and $\delta(\infty) =0$, whereby $\delta = (z\mapsto rz^{-1})$ for $r\in \C^\times$, or $\FP(\delta) = \{0,\infty\}$, which implies $\delta = (z\mapsto rz)$, $r\in \C^\times$.
Let $B'\subset \P^1$ be the set of branch points of $\phi'$. Then $\delta(B) = B'$. Since $1\in B'$, there exists $b\in B$, such that $\delta(b) =1$. This determines the factor $r$, and $\delta$ is one of the maps in the list. Conversely, every map in the list induces a covering map $\delta\phi$ of the desired form.
\end{proof}

We use the fact that the categories of compact Riemann surfaces (with non-constant holomorphic maps) and projective regular curves over $\C$ (with morphisms between them) are equivalent. From Lemma \ref{Section 2 - Lemma: Choice of covering map and parameters} and the general form of hyperelliptic curves, we directly get the following proposition.

\begin{prop}\label{Section 2 - Proposition: Form of curves with Property star}
The compact Riemann surfaces in $\Surface \subset M_2$ correspond bijectively to the isomorphism classes of affine plane curves $C_{\lambda,\mu}$ given by
\[v^2 = (u^2-1)(u^2-\lambda^2)(u^2-\mu^2)\]
where $(\lambda, \mu) \in P$ (and $P$ is as in Equation (\ref{Section 2 - Equation: Parameter space})). More precisely, for any such surface $Y$, there is a choice of parameters $(\lambda,\mu)\in P$, such that $Y$, considered as a projective regular curve, is birational to $C_{\lambda,\mu}$. Conversely, the associated compact Riemann surface to $C_{\lambda,\mu}$ defines a point in $\Surface$. The subgroup of $\Aut(C_{\lambda,\mu})$ in question is 
\[\{\id,\ (u,v)\mapsto (u,-v),\ (u,v)\mapsto(-u,v),\ (u,v)\mapsto (-u,-v)\},\]
where $(u,v)\mapsto (u,-v)$ is the hyperelliptic involution. The quotient map to $\P^1$ is given by $(u,v)\mapsto u$.
\end{prop}


\subsection{From the parameter space to the moduli space}
We investigate the map $\pr : P\ra M_2$ that sends $(\lambda,\mu) \in P$ to the isomorphism class of the curve $C_{\lambda,\mu}$. We come very close to $M_2$ itself by using a group action on the parameter space $P$ (see \cite{Gey74}).

\begin{prop}\label{Section 2 - Proposition: Surface S in the moduli space}
The group $\Gamma$ generated by
\[a:(\lambda,\mu)\mapsto (\lambda^{-1},\mu^{-1}),\qquad b:(\lambda,\mu)\mapsto (\mu,\lambda),\qquad c:(\lambda,\mu)\mapsto (\lambda^{-1},\lambda^{-1}\mu)\]
\[d:(\lambda,\mu)\mapsto (-\lambda,\mu),\qquad e:(\lambda,\mu)\mapsto (-\lambda,-\mu)\]
acts on the algebraic variety $P$ as a group of automorphisms. The following holds:
\begin{enumerate}[a)]
\item $\Gamma$ is isomorphic to the semidirect product $V_4\rtimes_\varphi D_6$, where the dihedral group $D_6 \isom \gen{a,b,c}$ acts on the Klein four group $V_4 \isom \gen{d,e}$ by conjugation.
\item The map $\pr:P\ra M_2$, $(\lambda,\mu)\mapsto C_{\lambda,\mu}$ induces a
surjective, birational morphism
\[\barpr: P/\Gamma \ra \Surface \subset M_2.\]
\item If we restrict $\barpr$ to $\barpr^{-1}((\Surface \setminus U) \cup \{Q\})$, where $U$ and $Q$ are defined as in Proposition \ref{Section 2 - Propositon: Overview of M_2}, then $\barpr$ is injective.
\end{enumerate}
\end{prop}

\begin{proof}
\textit{Part a).} Note that each of these maps is a well-defined automorphism of $P$. Clearly, $d^2= e^2 = \id$ and $de=ed$, so $\gen{d,e} \isom V_4$. Moreover, $a^2 = \id$, and one easily shows that $ab = ba$, $ac= ca$. The elements $b$ and $bc$ generate a subgroup isomorphic to $S_3$. Surely, $b^2 = \id$, and an easy computation shows that $bc$ has order 3 and that $b(bc) = (bc)^2b$. Therefore, 
\[\gen{a,b,c} \isom (\Z/2\Z)\times S_3 \isom D_6.\]
It remains to show that $\gen{d,e}$ is a normal subgroup of $\Gamma$. This can be verified on the generators:
$$ada = d, \ aea = e, \ bdb= de, \ beb = e,\ cdc = e,\ cec = d.$$
Thus, $\varphi:\gen{a,b,c}\ra \Aut(\gen{d,e}),\ g\mapsto (h\mapsto ghg^{-1})$ is a well-defined homomorphism and $\Gamma \isom V_4\rtimes_\varphi D_6$.

\textit{Part b).} By Proposition \ref{Section 2 - Proposition: Form of curves with Property star}, the map $\pr:P \ra \Surface$ is surjective. Next, we justify that $\pr:P\ra M_2$ factors through $P/\Gamma$. So let $(\lambda,\mu)\in P$ and $(\lambda',\mu') = \gamma \cdot(\lambda,\mu)$ for $\gamma\in\Gamma$. We have compact Riemann surfaces $Y$ and $Y'$ and degree 2-coverings $\phi:Y\ra P^1$, $\phi':Y'\ra\P^1$, which are branched over
\[B = \{1,-1,\lambda,-\lambda,\mu,-\mu\}\quad \txt{and}\quad B'=\{1,-1,\lambda',-\lambda',\mu',-\mu'\}\]
respectively. Now $Y\isom Y'$, if either we already have $B=B'$ or if there is $\delta$ from the list in Lemma \ref{Section 2 - Lemma: Choice of covering map and parameters}, such that $B' = \delta(B)$. Indeed, the affine curve $C_{\lambda,\mu}$ together with the covering $(u,v)\mapsto u$ is uniquely determined by the set of its branch points, and this covering is unique up to composition with such a $\delta$. Applying $\gamma\in \gen{b,d,e}$ to $(\lambda,\mu)$ does not affect the correspondig set $B$, so there is nothing to show. Applying $a$ to $(\lambda,\mu)$ corresponds to composing the covering $\phi:Y\ra \P^1$ with $\delta = (z\mapsto z^{-1})$, so $C_{\lambda,\mu} \isom C_{\lambda^{-1},\mu^{-1}}$. In the same way, $c$ corresponds to $\delta = (z\mapsto \lambda^{-1}z)$. This shows that we get a map $\barpr:P/\Gamma \ra M_2$. The birationality of $\barpr$ follows from Part c), since $\Surface\setminus U$ is open in $\Surface$.

\textit{Part c).} Let $(\lambda_1,\mu_1)$, $(\lambda_2,\mu_2)\in P$, and let $Y_i$ be the compact Riemann surface associated to $C_{\lambda_i,\mu_i}$, $i=1,2$. Let
\[\Aut(Y_i) \supseteq \{\id,\sigma_i,\tau_i,\sigma_i\tau_i\} \isom V_4,\]
where $\sigma_i$ is the hyperelliptic involution on $Y_i$, and $\tau_i$ is another involution with $\FP(\sigma_i)\cap \FP(\tau_i) = \emptyset$. Let $\phi_i:Y_i\ra \P^1$ be the associated covering map, coming from the projection onto the first coordinate of $C_{\lambda_i,\mu_i}$. Then $B_i = \{1,-1,\lambda_i,-\lambda_i,\mu_i,-\mu_i\}$ is the set of branch points of $\phi_i$.

Suppose that $Y_1$, $Y_2 \in (\Surface\setminus U) \cup \{Q\}$ and that we have an isomorphism $h:Y_1\ra Y_2$. We show that there is an isomorphism $h':Y_1\ra Y_2$ such that the covering map $\phi' = \phi_2 h':Y_1\ra \P^1$ satisfies the hypothesis of Lemma \ref{Section 2 - Lemma: Choice of covering map and parameters} b). Hence there exists a $\delta$ in List \ref{Section 2 - Equation: List of deltas} with $\delta \phi_1 = \phi_2 h'$. Thus, the branch points of $\phi_1$ are altered by $\delta$ from List \ref{Section 2 - Equation: List of deltas}, and one easily shows that if $\delta(B_1) = B_2$, then there is $\gamma\in \Gamma$ such that $\gamma\cdot (\lambda_1,\mu_1) = (\lambda_2,\mu_2)$.

To begin with, note that $\phi_2 h$ is also a quotient map for the hyperelliptic involution $\sigma_1$ on $Y_1$: the map $h\sigma_1 h^{-1}$ is a holomorphic involution on $Y_2$ with $6$ fixed points, and it follows from the uniqueness of the hyperelliptic involution, that $h\sigma_1 h^{-1} = \sigma_2$.
Next, consider the map $h\tau_1h^{-1}\in \Aut(Y_2)$. It descends to some involution $\tilde{\tau} \in \bar{\Aut(Y_2)}$. By Proposition \ref{Section 2 - Propositon: Overview of M_2}, we know the automorphism groups explicitly; in particular, $Y_i\in (\Surface\setminus U)\cup\{Q\}$ implies that there is only one conjugacy class of involutions in $\bar{\Aut(Y_i)}$. Therefore, the map $\tilde\tau$ is conjugate to $\bar\tau: z\mapsto -z$, the image of $\tau_2$ in $\bar{\Aut(Y_2)}$. So there exists some $\beta\in \bar{\Aut(Y_2)}$, such that $\beta\tilde\tau\beta^{-1} = \bar\tau$. The map $\beta$ has a lift $k\in \Aut(Y_2)$, and we set $h' = kh$. Then $\phi_2 h'$ is still a quotient map for $\sigma_1$, for which $\tau_1$ descends to $\bar\tau$, so Lemma \ref{Section 2 - Lemma: Choice of covering map and parameters} b) applies.
\end{proof}

It remains to study what happens if we restrict $\barpr : P/\Gamma \ra M_2$ to the curve $U$, resp. to $U\setminus \{Q\}$. A careful inspection leads to the following results (again we cite \cite{Gey74}).

\begin{prop}[see {\cite[Case 6]{Gey74}}]\label{Section 2 - Proposition: Where P mod Gamma to M_2 is not injective}
For a point $Y\in \Surface$, let $c(Y)$ be the number of conjugacy classes of involutions in $\bar{\Aut(Y)}$. Then $Y$ has precisely $c(Y)$ preimages in $P/\Gamma$. In particular, the map $\barpr: \barpr^{-1}(U\setminus \{Q\}) \ra M_2$ is 2 to 1. The two preimages of the point $Q'\in U$ are the $\Gamma$-orbits
\[\Gamma\cdot (e^{2i\pi/3},e^{i\pi/3})\quad \txt{\textit{and}}\quad \Gamma\cdot (-2+\sqrt{3}, -2-\sqrt{3}).\]

Every point in $P$ with nontrivial stabilizer in $\Gamma$ is in the $\Gamma$-orbit of a point on the curve $F = \set{(\lambda,\mu)\in P}{\lambda\mu = 1}$, and we have $\pr(F) = U$.
\end{prop}

\subsection{Automorphisms of the affine curve}
Let $Y\in\Surface$, and let $C_{\lambda,\mu}$ be birational to $Y$. Again, we write $\phi:Y\ra\P^1$ for the covering coming from $(u,v)\mapsto u$ and $\{\id,\sigma,\tau,\sigma\tau\}$ for the automorphism group of a generic point $Y\in\Surface$. We now take a look at the automorphisms $\tau$ and $\sigma\tau$. Inspecting the Riemann-Hurwitz formula, we find that $\tau$ and $\sigma\tau$ both have two fixed points that form a $\gen{\sigma}$-orbit (since we can assume that $\FP(\tau)\cap\FP(\sigma\tau) = \emptyset$ by Lemma \ref{Section 2 - Lemma: wlog sigma and tau have no common fixed point}). The maps $\tau$ and $\sigma\tau$ induce the automorphisms
\[(u,v)\mapsto (-u,v) \text{ and } (u,v)\mapsto (-u,-v)\]
of $C_{\lambda,\mu}$. 

\begin{rem}\label{Section 2 - Remark: wlog tau is (u,v) mapsto (-u,v)}
Without loss of generality, the automorphism $\tau$ (resp. $\sigma\tau$) corresponds to $(u,v)\mapsto (-u,v)$ (resp. $(u,v)\mapsto (-u,-v)$) on $C_{\lambda,\mu}$.
\end{rem}

\begin{proof}
If this is not the case, then $\FP(\tau) = \phi^{-1}(\infty)$ and $\FP(\sigma\tau) = \phi^{-1}(0)$. Composing the covering $\phi$ with $z\mapsto z^{-1}$ leads to exchanging the fixed points of $\bar\tau$. But this in turn corresponds to replacing $(\lambda,\mu)$ with $(\lambda^{-1},\mu^{-1})$, which both lead to isomorphic curves by Proposition \ref{Section 2 - Proposition: Surface S in the moduli space}.
\end{proof}

In the following, we assume that $(\lambda,\mu)$ be chosen such that $\tau$ is given by $(u,v)\mapsto (-u,v)$.

\begin{cor}\label{Section 2 - Corollary: Fixed points of automorphisms}
The fixed points of $\tau$ correspond to the points $(0,i\lambda\mu)$, $(0,-i\lambda\mu) \in C_{\lambda,\mu}$ and the set of fixed points of $\sigma\tau$ is $\phi^{-1}(\infty) =\{\infty_1,\infty_2\}$.
\end{cor}

\begin{cor}\label{Section 2 - Corollary: Form of elliptic curve}
The quotient surface $\bar Y = Y/\gen{\tau}$ is the elliptic curve
\[y^2 = (x-1)(x-\lambda^2)(x-\mu^2)\]
with origin $N = \pi(\infty_1)=\pi(\infty_2)$. The quotient map $\pi:Y\ra \bar Y$ is given by $(u,v)\mapsto (x,y)=(u^2,v)$ on $C_{\lambda,\mu}$.
\end{cor}

\section{Automorphism groups of translation surfaces in genus 2}
Recall from Section \ref{Section 1 - Translation surfaces}, that a compact Riemann surface $X$ together with a non-zero holomorphic 1-form $\omega$ on $X$ defines a translation surface $(X,\omega)$. In this section, we answer the question how the automorphism group $\Aut(X,\omega)$ can possibly look like, if $X$ has genus 2.

\begin{prop}\label{Section 3 - Proposition: Automorphisms of translation surfaces in g=2}
Let $(X,\omega)$ be a translation surface of genus 2. Then either 
\[\Aut(X,\omega) =\{\id, \sigma\} \isom \Z/2\Z\quad \text{or} \quad \Aut(X,\omega) = \{\id,\sigma,\tau,\sigma\tau\}\isom V_4,\]
where $\sigma$ is the hyperelliptic involution on $X$ and $\tau$ is a translation of order 2. Moreover, if $(X,\omega)$ is in the stratum $\Stratum(2)$, then only the first case is possible.
\end{prop}
\begin{proof} First note that the hyperelliptic involution $\sigma$ is a biholomorphic map that lives on the whole Teichmüller disk to $(X,\omega)$. So by Proposition \ref{Section 1 - Proposition: common automorphisms of TM-disk are affine}, it is affine with derivative $-I$. Indeed, $\sigma$ cannot be a translation, for it has 6 fixed points, and thus cannot act freely on the translation surface $(X,\omega)$. One has an exact sequence
\[1\rTo \Trans(X,\omega) \rTo \Aut(X,\omega) \rTo^{\der} \{\pm I\} \rTo 1,\]
thus it suffices to determine $\Trans(X,\omega)$ to prove the claim. To this end, we distinguish two cases. First, let $(X,\omega)$ be in the stratum $\Stratum(1,1)$ and let $\tau$ be a translation. Then $\tau$ is a fortiori a biholomorphic automorphism, thus it has finite order. Moreover, $\tau$ permutes the two zeros $P$, $P'$ of $\omega$. We look at the quotient surface $X/\gen{\tau}$. Suppose that $\tau$ does not fix $P$ and $P'$. Then it has no fixed point in $X$ (since it is a translation on $X\setminus \{P,P'\}$). By Riemann-Hurwitz, $g' = g(X/\gen{\tau}) \leq g(X) = 2$ and
\[2g(X) - 2 = 2 = \ord(\tau)(2g' -2).\]
Each of the cases $g' \in \{0,1,2\}$ leads to a contradiction. So $\FP(\tau) = \{P, P'\}$. Let $\gamma$ be a geodesic for the translation structure on $X$ that starts from a singularity, say $P$, in direction $v$. Since the cone angle around $P$ is equal to $4\pi$, there are precisely two such geodesics $\gamma_1$, $\gamma_2$. A translation $\tau$ that fixes $P$ permutes $\gamma_1$ and $\gamma_2$. By the identity theorem, $\tau = \id$ or $\tau^2 = \id$, and this determines $\tau$ uniquely.

The case $(X,\omega) \in \Stratum(2)$ is treated in \cite[Proposition 4.4]{HL06}.
\end{proof}


\section{The origami $S$}
In this section, we study the origami $S$ and its origami curve $\Curve_S$ in the moduli space. Step by step, we prove the assertions of Theorem \ref{Introduction - Thm - Main Result}.

\subsection{Automorphisms of $S$}
We write $S=(p:X\ra E)$ for the origami covering. Let $X_I = (X,p^*\omega_I)$ be the square-tiled translation surface defined by $S$. We determine the group $\Aut(X_I)$ of common automorphisms of points in the Teichmüller disk $\Delta_S = \Delta(X,(p^*\omega_I)^2))$.

Observe that Figures \ref{Figure: sigma} and \ref{Figure: tau} define two affine diffeomorphisms of $X_I$, which we call $\sigma$ and $\tau$. We have $\der(\sigma) = -I$ and $\der(\tau) = I$, so in particular $\tau$ is a translation. It also follows from the pictures that $\sigma$ and $\tau$ are both of order 2. Their product $\sigma\tau$ (which is the same as $\tau\sigma$) is depicted in Figure \ref{Figure: sigmatau}.

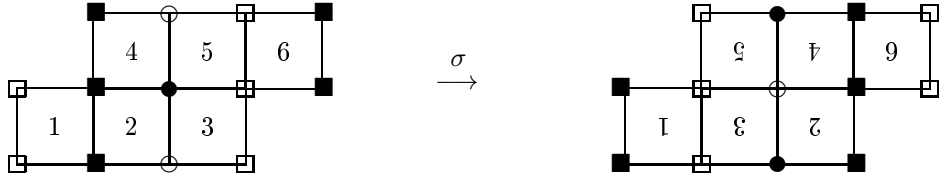
\begin{figure}[ht]
\setlength{\unitlength}{1cm}
\begin{picture}(12,2.5)
\put(5.5,1){$\longrightarrow$}
\put(5.7,1.3){$\sigma$}

\put(0,0){\framebox(1,1){1}}
\put(1,0){\framebox(1,1){2}}
\put(2,0){\framebox(1,1){3}}
\put(1,1){\framebox(1,1){4}}
\put(2,1){\framebox(1,1){5}}
\put(3,1){\framebox(1,1){6}}

\put(-.1,-0.1){\framebox(.2,.2)}
\put(-.1,0.9){\framebox(.2,.2)}
\put(2.9,-0.1){\framebox(.2,.2)}
\put(2.9,0.9){\framebox(.2,.2)}
\put(2.9,1.9){\framebox(.2,.2)}

\put(.9,-0.1){$\blacksquare$}
\put(.9,0.9){$\blacksquare$}
\put(.9,1.9){$\blacksquare$}
\put(3.9,0.9){$\blacksquare$}
\put(3.9,1.9){$\blacksquare$}

\put(2,0){\circle{0.2}}
\put(2,1){\circle*{0.2}}
\put(2,2){\circle{0.2}}

\put(8,0){\framebox(1,1){\begin{turn}{180}  1 \end{turn}}}
\put(9,0){\framebox(1,1){\begin{turn}{180}  3 \end{turn}}}
\put(10,0){\framebox(1,1){\begin{turn}{180} 2 \end{turn}}}
\put(9,1){\framebox(1,1){\begin{turn}{180}  5 \end{turn}}}
\put(10,1){\framebox(1,1){\begin{turn}{180} 4 \end{turn}}}
\put(11,1){\framebox(1,1){\begin{turn}{180} 6 \end{turn}}}

\put(7.8,-0.1){$\blacksquare$}
\put(7.8,0.9){$\blacksquare$}
\put(10.9,-0.1){$\blacksquare$}
\put(10.9,0.9){$\blacksquare$}
\put(10.9,1.9){$\blacksquare$}

\put(8.9,-0.1){\framebox(.2,.2)}
\put(8.9,0.9){\framebox(.2,.2)}
\put(8.9,1.9){\framebox(.2,.2)}
\put(11.9,0.9){\framebox(.2,.2)}
\put(11.9,1.9){\framebox(.2,.2)}

\put(10,0){\circle*{0.2}}
\put(10,1){\circle{0.2}}
\put(10,2){\circle*{0.2}}
\end{picture}
\caption{The hyperelliptic involution $\sigma$ on the origami $S$\label{Figure: sigma}}
\end{figure}

\begin{figure}[ht]
\setlength{\unitlength}{1cm}
\begin{picture}(12,2.5)
\put(5.5,1){$\longrightarrow$}
\put(5.7,1.3){$\tau$}

\put(0,0){\framebox(1,1){1}}
\put(1,0){\framebox(1,1){2}}
\put(2,0){\framebox(1,1){3}}
\put(1,1){\framebox(1,1){4}}
\put(2,1){\framebox(1,1){5}}
\put(3,1){\framebox(1,1){6}}

\put(-.1,-0.1){\framebox(.2,.2)}
\put(-.1,0.9){\framebox(.2,.2)}
\put(2.9,-0.1){\framebox(.2,.2)}
\put(2.9,0.9){\framebox(.2,.2)}
\put(2.9,1.9){\framebox(.2,.2)}

\put(.9,-0.1){$\blacksquare$}
\put(.9,0.9){$\blacksquare$}
\put(.9,1.9){$\blacksquare$}
\put(3.9,0.9){$\blacksquare$}
\put(3.9,1.9){$\blacksquare$}

\put(2,0){\circle{0.2}}
\put(2,1){\circle*{0.2}}
\put(2,2){\circle{0.2}}

\put(8,0){\framebox(1,1){6}}
\put(9,0){\framebox(1,1){4}}
\put(10,0){\framebox(1,1){5}}
\put(9,1){\framebox(1,1){2}}
\put(10,1){\framebox(1,1){3}}
\put(11,1){\framebox(1,1){1}}

\put(7.9,-0.1){\framebox(.2,.2)}
\put(7.9,0.9){\framebox(.2,.2)}
\put(10.9,-0.1){\framebox(.2,.2)}
\put(10.9,0.9){\framebox(.2,.2)}
\put(10.9,1.9){\framebox(.2,.2)}

\put(8.9,-0.1){$\blacksquare$}
\put(8.9,0.9){$\blacksquare$}
\put(8.9,1.9){$\blacksquare$}
\put(11.9,0.9){$\blacksquare$}
\put(11.9,1.9){$\blacksquare$}

\put(10,0){\circle*{0.2}}
\put(10,1){\circle{0.2}}
\put(10,2){\circle*{0.2}}
\end{picture}
\caption{The translation $\tau$ on the origami $S$\label{Figure: tau}}
\end{figure}
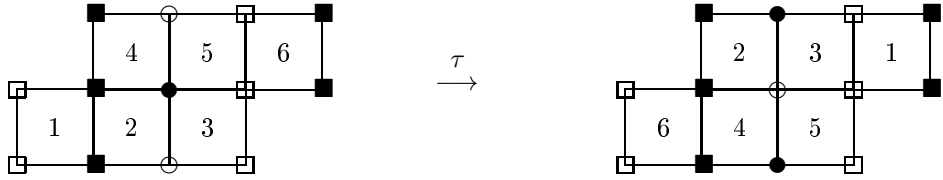

\begin{figure}[ht]
\setlength{\unitlength}{1cm}
\begin{picture}(12,2.5)
\put(5.5,1){$\longrightarrow$}
\put(5.6,1.3){$\sigma\tau$}

\put(0,0){\framebox(1,1){\upshape 1}}
\put(1,0){\framebox(1,1){\upshape 2}}
\put(2,0){\framebox(1,1){\upshape 3}}
\put(1,1){\framebox(1,1){\upshape 4}}
\put(2,1){\framebox(1,1){\upshape 5}}
\put(3,1){\framebox(1,1){\upshape 6}}

\put(-.1,-0.1){\framebox(.2,.2)}
\put(-.1,0.9){\framebox(.2,.2)}
\put(2.9,-0.1){\framebox(.2,.2)}
\put(2.9,0.9){\framebox(.2,.2)}
\put(2.9,1.9){\framebox(.2,.2)}

\put(.9,-0.1){$\blacksquare$}
\put(.9,0.9){$\blacksquare$}
\put(.9,1.9){$\blacksquare$}
\put(3.9,0.9){$\blacksquare$}
\put(3.9,1.9){$\blacksquare$}

\put(2,0){\circle{0.2}}
\put(2,1){\circle*{0.2}}
\put(2,2){\circle{0.2}}

\put(8,0){\framebox(1,1){\begin{turn}{180}\upshape  6 \end{turn}}}
\put(9,0){\framebox(1,1){\begin{turn}{180}\upshape  5 \end{turn}}}
\put(10,0){\framebox(1,1){\begin{turn}{180}\upshape 4 \end{turn}}}
\put(9,1){\framebox(1,1){\begin{turn}{180}\upshape  3 \end{turn}}}
\put(10,1){\framebox(1,1){\begin{turn}{180}\upshape 2 \end{turn}}}
\put(11,1){\framebox(1,1){\begin{turn}{180}\upshape 1 \end{turn}}}

\put( 7.8,-0.1){$\blacksquare$}
\put( 7.8, 0.9){$\blacksquare$}
\put(10.9,-0.1){$\blacksquare$}
\put(10.9, 0.9){$\blacksquare$}
\put(10.9, 1.9){$\blacksquare$}

\put( 8.9,-0.1){\framebox(.2,.2)}
\put( 8.9, 0.9){\framebox(.2,.2)}
\put( 8.9, 1.9){\framebox(.2,.2)}
\put(11.9, 0.9){\framebox(.2,.2)}
\put(11.9, 1.9){\framebox(.2,.2)}

\put(10,0){\circle{0.2}}
\put(10,1){\circle*{0.2}}
\put(10,2){\circle{0.2}}
\end{picture}
\caption{The map $\sigma\tau$ on the origami $S$\label{Figure: sigmatau}}
\end{figure}
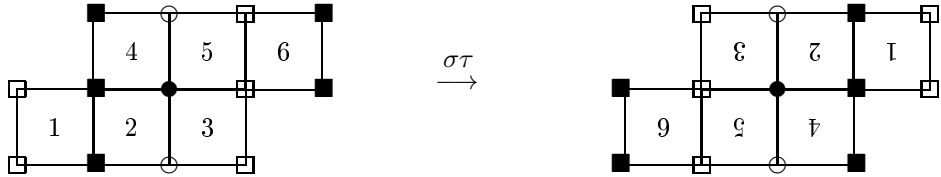

\begin{prop}\label{Section 4 - Proposition: Aut(X,omega) for S}
The group $\Aut(X_I)$ is given by $\{\id,\sigma,\tau,\sigma\tau\} \isom V_4$. The map $\sigma$ is the hyperelliptic involution and we have $\Trans(X_I) = \{\id,\tau\}$.
\end{prop}
\begin{proof}
Clearly, $\Aut(X_I)$ contains these elements; moreover, by Proposition \ref{Section 3 - Proposition: Automorphisms of translation surfaces in g=2}, $\Aut(X_I)$ cannot be bigger, and $\tau$ is the non-trivial translation. Furthermore, observe that the map $\sigma$ is an involution having the six fixed points as indicated in Figure \ref{Figure: fixed points of sigma}. So $\sigma$ is the hyperelliptic involution (since the latter is unique).
\end{proof}

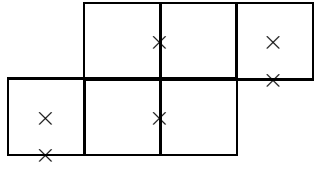
\begin{figure}[ht]
\setlength{\unitlength}{1cm}
\begin{picture}(4,2)
\put(0,0){\framebox(1,1){}}
\put(1,0){\framebox(1,1){}}
\put(2,0){\framebox(1,1){}}
\put(1,1){\framebox(1,1){}}
\put(2,1){\framebox(1,1){}}
\put(3,1){\framebox(1,1){}}

\put(0.35,0.4){$\times$}
\put(0.35,-0.1){$\times$}
\put(1.85,0.4){$\times$}

\put(3.35,1.4){$\times$}
\put(3.35,0.9){$\times$}
\put(1.85,1.4){$\times$}
\end{picture}
\caption{Fixed points of $\sigma$\label{Figure: fixed points of sigma}}
\end{figure}

Note that the points $\square$ and $\blacksquare$ are singularities on the translation surface $X_I$. Consequently, $X_I$ is in the stratum $\Stratum(1,1)$. Next, we identify the fixed points of the remaining automorphisms.

\begin{rem}
The fixed points of the translation $\tau$ are $\square$ and $\blacksquare$. The fixed points of the map $\sigma\tau$ are $\tcircle$ and $\tcircleblack$. 
\end{rem}

\subsection{An equation for the points on $\Curve_S$}
By Proposition \ref{Section 4 - Proposition: Aut(X,omega) for S} and Proposition \ref{Section 2 - Proposition: Form of curves with Property star}, we know that every surface on the Teichmüller disk associated to $S$ corresponds to a curve $C_{\lambda,\mu}$ for some $(\lambda, \mu) \in P$. Our next task is to find an algebraic relation between $\lambda$ and $\mu$ that describes the points in $M_2$ that lie on the origami curve $\Curve_S$. 

The map $\tau:X\ra X$ is a deck transformation for $S = (p:X\ra E)$. Let $\pi:X\ra X/\gen{\tau}$ denote the quotient map. The map $p$ factors as $\bar p\circ \pi$, where $\bar S = (\bar p: X/\gen{\tau} \ra E)$ is an origami of genus one. A picture of $\bar S$ is drawn in Figure~\ref{Figure: Origami S_quer}. Gluings are made by identifying opposite sides and the numbers indicate which squares of $S$ are identified by the action of $\tau$. The points $\lozenge =\pi(\square)$, $\blacklozenge = \pi(\blacksquare)$ are the images of the fixed points of $\tau$ and $@$ is the image of $\{\tcircle,\tcircleblack\}$ under $\pi$.

\begin{figure}[th]
\setlength{\unitlength}{1.5cm}
\begin{center}
\begin{picture}(3,1)
\put(0,0){\framebox(1,1){3(5)}}
\put(1,0){\framebox(1,1){1(6)}}
\put(2,0){\framebox(1,1){2(4)}}

\put(0.92,-.08){$\lozenge$}
\put(0.92,.92){$\lozenge$}

\put(1.92,-.09){$\blacklozenge$}
\put(1.92,.91){$\blacklozenge$}

\put(-.1,-.1){$@$}
\put(-.1,.9){$@$}
\put(2.9,-.1){$@$}
\put(2.9,.9){$@$}
\end{picture}\\[1ex]
\end{center}
\caption{The origami $\bar S$ \label{Figure: Origami S_quer}}
\end{figure}
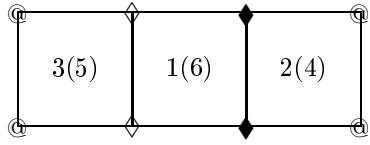


Let $(X,\omega)$ be a point on the Teichmüller disk $\Delta_S$ of the origami $S$. The differential $\omega = p^*\omega_A$ is the pullback of the differential $\omega_A$ on the complex torus $E_A$ for some $A\in\SL_2(\R)$. We write $X_A$ for the Riemann surface defined by $(X,\omega) = (X,p^*\omega_A)$. Then $\bar X_A = X_A/\gen{\tau}$ is an elliptic curve, where we choose the origin $@$.

It is clear from Figure \ref{Figure: Origami S_quer} that we can identify $(\bar X_A,@)$ with the elliptic curve $E_{AB}$, where $B = \textmatrix{3}{0}{0}{1}$ (we assume that $E_{AB}$ is equipped with the group structure that descends from $\C$). In this way, the points $\lozenge$ and $\blacklozenge$ correspond to the points $1+\Lambda_{AB}$, respectively $2+\Lambda_{AB}$. 

\begin{rem}\label{Section 4 - Remark: we have 3-torsion points}
The points $\lozenge$ and $\blacklozenge$ are 3-torsion points of the elliptic curve $(\bar X_A,@)$ and their sum equals $@$.
\end{rem}





With these considerations in mind, we can prove the first part of Theorem \ref{Introduction - Thm - Main Result}. 

\begin{prop}\label{Section 4 - Proposition: Equation of Curve_S}
The origami curve $\Curve_S$ of the origami $S$ is equal to the projection of the affine curve $V\subset \C^2$ to the moduli space $M_2$, where
\[V:\quad \mu(\lambda +1) - \lambda = 0,\quad \lambda \neq 0, \pm 1, -\tfrac{1}{2}, -2.\]
Explicitly, every point on $\Curve_S$ is birational as affine plane curve to
$$y^2 = (x^2-1)(x^2-\lambda^2)(x^2-\bigl(\tfrac{\lambda}{\lambda+1}\bigr)^2)$$
for some $\lambda \in \C\setminus \{0,\pm 1, -\tfrac{1}{2},-2\}$.
\end{prop}

\begin{proof}
As seen above, every point on $\Curve_S$ is represented by a translation surface $X_A$ for some $A\in \SL_2(\R)$. By Proposition \ref{Section 4 - Proposition: Aut(X,omega) for S}, we have $\Curve_S\subset \Surface$, so it follows from Proposition \ref{Section 2 - Proposition: Form of curves with Property star} that there is a covering map $\phi:X_A\ra \P^1$, ramified over $B=\{1,-1,\lambda,-\lambda,\mu,-\mu\}$ for some parameters $(\lambda,\mu)\in P$ such that $X_A$ is birational to the affine plane curve $C_{\lambda,\mu}$. By Remark \ref{Section 2 - Remark: wlog tau is (u,v) mapsto (-u,v)}, we can assume that $(\lambda,\mu)\in P$ are chosen such that $\tau\in \Aut(X_A)$ corresponds to the morphism $(u,v)\mapsto (-u,v)$ of $C_{\lambda,\mu}$. By Corollary \ref{Section 2 - Corollary: Form of elliptic curve}, $(\bar X_A,@)$ is isomorphic to
\[y^2 = (x-1)(x-\lambda^2)(x-\mu^2)\]
with origin at infinity, and Corollary \ref{Section 2 - Corollary: Fixed points of automorphisms} implies that
\[\{P_1 = (0,i\lambda\mu),P_2 = (0,-i\lambda\mu)\} = \{\lozenge, \blacklozenge\},\]
since they are the images of the fixed points of $\tau$. We use the addition formula for points on an elliptic curve to derive a relation between $\lambda$ and $\mu$ (see \cite[III.2.3]{Sil92}). $(\bar X_A,@)$ has the data
\begin{align*}
a_1 &= 0 & a_3 &= 0\\
a_2 &= -1-\lambda^2-\mu^2 & a_4 &= \lambda^2 + \mu^2 + \lambda^2\mu^2\\
a_6 &= -\lambda^2\mu^2 & &
\end{align*}
By Remark \ref{Section 4 - Remark: we have 3-torsion points}, we have that
\[[3]P_1 = [3]P_2 = \infty,\]
and
\[[2]P_1 = P_2.\]
So we compute the double of $P_1$ with respect to the group structure on $(\bar X_A,@)$ and compare it with $P_2$. Let $P_1 = (x_1,y_1) = (0,i\lambda\mu)$ and $P_2 = (x_2,y_2) = (0,-i\lambda\mu)$. Let
\[
\alpha  = \frac{3x_1^2 + 2a_2x_1 + a_4 - a_1y_1}{2y_1 + a_1x_1 + a_3}\\
        = \frac{\lambda^2+\mu^2+\lambda^2\mu^2}{2i\lambda\mu}
\]
and
\[
\beta   = \frac{-x_1^3 + a_4x_1 + 2a_6 - a_3y_1}{2y_1 + a_1x_1 + a_3}\\
        = \frac{-2\lambda^2\mu^2}{2i\lambda\mu} = i\lambda\mu.
\]
Then by \cite[III.2.3]{Sil92}, the $x$-coordinate of $[2]P_1$ is given by
\beq
x([2]P_B) &=& \alpha^2 + a_1\alpha - a_2 - x_1 - x_1\\ 
        &=& \left(\frac{\lambda^2+\mu^2+ \lambda^2\mu^2}{2i\lambda\mu}\right)^2 + 1 + \lambda^2+\mu^2\\
        &=& \frac{\lambda^4+\mu^4+\lambda^4\mu^4 + 2\lambda^2\mu^2+ 2\lambda^4\mu^2+2\lambda^2\mu^4}{-4\lambda^2\mu^2} + 1 + \lambda^2 + \mu^2\\
        &=& \frac{1}{-4\lambda^2\mu^2}\Big(\lambda^4+\mu^4+\lambda^4\mu^4+2\lambda^2\mu^2+2\lambda^4\mu^2+2\lambda^2\mu^4 - {}\\
        & & {} - 4\lambda^2\mu^2 - 4\lambda^4\mu^2 - 4\lambda^2\mu^4\Big)\\
        &=& \frac{1}{-4\lambda^2\mu^2}\Big(\lambda^4+\mu^4+\lambda^4\mu^4 - 2\lambda^2\mu^2 - 2\lambda^4\mu^2- 2\lambda^2\mu^4\Big)\\
        &=& \frac{1}{-4\lambda^2\mu^2}\Big((-\lambda^2-\mu^2+\lambda^2\mu^2)^2-4\lambda^2\mu^2\Big)\\
        &=& \left(\frac{-\lambda^2-\mu^2+\lambda^2\mu^2}{2i\lambda\mu}\right)^2 + 1
\eeq
Since $x([2]P_1) = x_2 = 0$, it follows that
\[0 = \left(\frac{-\lambda^2-\mu^2+\lambda^2\mu^2}{2i\lambda\mu}\right)^2 + 1,\]
and therefore
\[\frac{-\lambda^2-\mu^2+\lambda^2\mu^2}{2i\lambda\mu} = \pm i,\]
or equivalently
\[-\lambda^2 - \mu^2 + \lambda^2\mu^2 = \mp 2\lambda \mu.\]
We distinguish two cases:\\[2ex]
\underline{Case 1:} $-\lambda^2 - \mu^2+ \lambda^2\mu^2 = - 2\lambda\mu$\\
Then one has
\[- \lambda^2 - \mu^2 + 2\lambda\mu = -(\lambda - \mu)^2 = -\lambda^2\mu^2,\]
hence
\[\lambda - \mu  \pm \lambda \mu = 0.\]
\underline{Case 2:} $-\lambda^2 - \mu^2+ \lambda^2\mu^2 = + 2\lambda\mu$\\
Then one has
\[- \lambda^2 - \mu^2 - 2\lambda\mu = -(\lambda + \mu)^2 = -\lambda^2\mu^2,\]
which implies
\[\lambda + \mu \pm \lambda \mu =0.\]
Since $y([2]P_1)$ evaluates to
\beq
y([2]P_1) &=& -( \alpha + a_1)x([2]P_1) - \beta - a_3\\
        &=& -\beta\\
        &=& -i\lambda\mu = y_2,
\eeq
the equation $[2]P_1 = P_2$ is fulfilled, if $(\lambda,\mu)$ lies on one of the four affine curves 
\beqn \label{Equation: Vs} V_{\eps_1,\eps_2}:\quad \eps_2\mu(\eps_1\lambda + 1) - \eps_1\lambda &=& 0.\eeqn
with $\{\eps_1,\eps_2\}\in\{\pm 1\}^2$.
One easily checks that $P_1+P_2 = \infty$ in each of these cases. It remains to explain, why these four cases reduce to a single one. Recall from Proposition \ref{Section 2 - Proposition: Surface S in the moduli space} that the group $\Gamma$ acts on the parameter space $P$ and that the projection $P\ra M_2$ factors through $P/\Gamma$. We show that the curves $V_{\eps_1,\eps_2}$ are all in the same $\Gamma$-orbit, whence they are mapped to one curve in $M_2$:
\begin{align}\label{Equation: Action of d and e on Vs} d\cdot V_{\eps_1,\eps_2} = V_{-\eps_1,\eps_2},\qquad e\cdot V_{\eps_1,\eps_2} = V_{-\eps_1,-\eps_2}, \qquad ed\cdot V_{\eps_1,\eps_2} = V_{\eps_1,-\eps_2}.\end{align}
Let $V := V_{1,1}$. We set
\[F:\C\setminus\{0,\pm 1, -\tfrac{1}{2}, -2\}\ra P,\ \lambda\mapsto \left(\lambda,\frac{\lambda}{\lambda+1}\right).\]
This yields a well-defined injective morphism with image $V$. Therefore, every Riemann surface on the origami curve $\Curve_S$ is birational to a curve $C_{\lambda,\mu}$ with $(\lambda,\mu) \in V$. Thus, $\Curve_S \subset \pr(V)$, where $\pr:P\ra M_2$ is the projection from Proposition \ref{Section 2 - Proposition: Surface S in the moduli space}. Since the maps $F$ and $\pr$ are continuous for the Zariski-topology, since $\Curve_S$ is closed in $M_2$ and since $\C\setminus\{0,\pm 1, -\tfrac{1}{2}, -2\}$ is irreducible, we conclude that $\Curve_S = \pr(V)$.
\end{proof}

\subsection{Nonsingularity of $\Curve_S$}
As the origami curve $\Curve_S$ is the image of a complex geodesic in the Teichmüller space, it can have at most transverse self-intersections as singularities (see e.g. \cite[Proposition 2.10]{Lo03}). We show in the following that such self-intersections cannot occur, since there is a nonsingular curve that is mapped injectively onto $\Curve_S$.

Let $V = V_{1,1} \subset P$ be the affine curve defined in Equation \eqref{Equation: Vs}. Then $\pr(V) = \Curve_S$ by Proposition \ref{Section 4 - Proposition: Equation of Curve_S}. Observe that the curve $V$ is irreducible and regular. The subgroup $G:= \Stab_\Gamma(V)$ acts on $V$, and the quotient $V/G$ is again a nonsingular curve. Thus, the nonsingularity of $\Curve_S$ follows from the next proposition.

\begin{prop}\label{Section 4 - Proposition: V/G to M_2 is injective}
The nonsingular curve $V/\Stab_\Gamma(V)$ is mapped injectively onto $\Curve_S \subset M_2$.
\end{prop}

Note that $V$ (and consequently $V/G$) is affine of genus 0. Thus Proposition \ref{Section 4 - Proposition: V/G to M_2 is injective} implies the following corollary.

\begin{cor} \label{Section 4 - Corollary: Curve_S is regular}
The Teichmüller curve $\Curve_S$ is a regular, affine curve of genus 0.
\end{cor}

To prove Proposition \ref{Section 4 - Proposition: V/G to M_2 is injective}, we need to understand how $\Gamma$ acts on the curve $V$. This is essentially done in Lemma \ref{Section 4 - Lemma: Stabilizer of V}, which we are going to state now. 
First, have a look at the $\Gamma$-orbit of the irreducible component $V = V_{1,1}$. The action of the subgroup $\gen{d,e}$ has already been described in \eqref{Equation: Action of d and e on Vs}. The points on $b\cdot V_{\eps_1,\eps_2}$ satisfy the equation
\[\eps_2\lambda(\eps_1\mu + 1) - \eps_1\mu = 0,\]
which is equivalent to
\[-\eps_1\mu(-\eps_2\lambda + 1) - (-\eps_2\lambda) = 0,\]
so $b\cdot V_{\eps_1,\eps_2} = V_{-\eps_2,-\eps_1}$. For the element $c\in \Gamma$, we find that $c\cdot V_{\eps_1,\eps_2}$ is given by
\[\eps_2\lambda^{-1}\mu(\eps_1\lambda^{-1} + 1) - \eps_1\lambda^{-1} = 0.\]
We multiply this equation by $\lambda^2$ and get
\[\eps_1\eps_2\mu(\eps_1\lambda + 1) - \eps_1\lambda = 0.\]
Thus,
\[c \cdot V_{\eps_1,\eps_2} = V_{\eps_1,\eps_1\eps_2}.\]
Finally, applying $a\in \Gamma$ to $V_{\eps_1,\eps_2}$ yields
\[\eps_2\mu^{-1}(\eps_1\lambda^{-1} + 1) - \eps_1\lambda^{-1} = 0.\]
By multiplying this equation with $\eps_1\lambda\eps_2\mu$, we can rewrite it as
\begin{align} \label{Equation: Ws} 1 + \eps_1 \lambda - \eps_2 \mu = 0.\end{align}
Let the curves $W_{\eps_1,\eps_2}$ be defined by these equations, \ie $W_{\eps_1,\eps_2} = a\cdot V_{\eps_1,\eps_2}$. Since $a$ is a central element of $\Gamma$, the subgroup $\gen{b,c,d,e}$ of $\Gamma$ acts on the $W$'s as it acts on the $V$'s.

\begin{lem} \label{Section 4 - Lemma: Stabilizer of V}
\begin{enumerate}[a)]
\item The $\Gamma$-orbit of $V = V_{1,1}$ consists of the set of 8 irreducible components
\[M = \set{V_{\eps_1,\eps_2}}{(\eps_1,\eps_2) \in \{\pm 1\}^2}\ \cup\ \set{W_{\eps_1,\eps_2}}{(\eps_1,\eps_2)\in \{\pm 1\}^2},\]
and $\Gamma$ acts on $M$ as follows:
\begin{align*}
a\cdot V_{\eps_1,\eps_2} &= W_{\eps_1,\eps_2} & b\cdot V_{\eps_1,\eps_2} &= V_{-\eps_2,-\eps_1} & c\cdot V_{\eps_1,\eps_2} &= V_{\eps_1,\eps_1\eps_2}\\
d\cdot V_{\eps_1,\eps_2} &= V_{-\eps_1,\eps_2}& e\cdot V_{\eps_1,\eps_2} &= V_{-\eps_1,-\eps_2} & &\\
a\cdot W_{\eps_1,\eps_2} &= V_{\eps_1,\eps_2} & b\cdot W_{\eps_1,\eps_2} &= W_{-\eps_2,-\eps_1} & c\cdot W_{\eps_1,\eps_2} &= W_{\eps_1,\eps_1\eps_2}\\
d\cdot W_{\eps_1,\eps_2} &= W_{-\eps_1,\eps_2}& e\cdot W_{\eps_1,\eps_2} &= W_{-\eps_1,-\eps_2}. & &
\end{align*}
\item The stabilizer of $V$ in $\Gamma$ is the subgroup
\[G = \gen{c, dbd} \isom S_3,\]
and the $G$-orbit of a point $(\lambda,\mu) \in V$ is
\begin{align*}
(\lambda,\mu)&,& c\cdot(\lambda,\mu) &= (\lambda^{-1},\lambda^{-1}\mu),\\
cdbd\cdot(\lambda,\mu) &= (-\mu^{-1}, \lambda\mu^{-1}), & dbd\cdot(\lambda,\mu) &= (-\mu,-\lambda),\\
(cdbd)^2\cdot(\lambda,\mu) &= (-\lambda^{-1}\mu,-\lambda^{-1}), & cdbdc\cdot (\lambda,\mu) &= (-\lambda\mu^{-1},\mu^{-1})
\end{align*}
\item $V_{\eps_1,\eps_2} \cap V_{\eps_3,\eps_4} = \emptyset$ for $(\eps_1,\eps_2) \neq (\eps_3,\eps_4)$, and $|V_{\eps_1,\eps_2} \cap W_{\eps_3,\eps_4}| = 2$ for any $(\eps_1,\eps_2), (\eps_3,\eps_4) \in \{\pm 1\}^2$.
\item The points on $V$ with nontrivial stabilizer in $\Gamma$ form two $G$-orbits:
\[\left\{q_1 = (e^{2i\pi/3}, e^{i\pi/3}), q_2 = (e^{-2i\pi/3}, e^{-i\pi/3})\right\}\]
is a $G$-orbit with $\Stab_G(q_i) = \gen{cdbd}$ ($i=1,2$). The corresponding $\Gamma$-orbit has 8 elements and $\Stab_\Gamma(q_i) = \gen{ca,cdbd}\isom S_3$. The $G$-orbit
\begin{align*}
G\cdot\Bigl(\tfrac{-1+\sqrt{5}}{2}, \tfrac{3-\sqrt{5}}{2}\Bigr) = \biggl\{
	&r_1 = \Bigl(\tfrac{-1+\sqrt{5}}{2}, \tfrac{3-\sqrt{5}}{2}\Bigr), r_2 = \Bigl(\tfrac{-1-\sqrt{5}}{2}, \tfrac{3+\sqrt{5}}{2}\Bigr),\\
	&r_3 = \Bigl(\tfrac{1+\sqrt{5}}{2}, \tfrac{-1+\sqrt{5}}{2}\Bigr), r_4 = \Bigl(\tfrac{1-\sqrt{5}}{2}, \tfrac{-1-\sqrt{5}}{2}\Bigr)\\
	&r_5 = \Bigl(\tfrac{-3+\sqrt{5}}{2}, \tfrac{1-\sqrt{5}}{2}\Bigr), r_6 = \Bigl(\tfrac{-3-\sqrt{5}}{2}, \tfrac{1+\sqrt{5}}{2}\Bigr)\biggl\}
\end{align*}
corresponds to a $\Gamma$-orbit with 24 elements and $\Stab_\Gamma(r_i) \isom \Z/2\Z$ ($i=1,\dots,6$).
\end{enumerate}
\end{lem}
\begin{proof}
\textit{a) and b).} Note that the curves $V_{\eps_1,\eps_2}$ and $W_{\eps_1,\eps_2}$ are all irreducible and distinct, so $M$ is an 8-element set, and we have seen above that $\Gamma$ acts transitively on $M$. Since $\Gamma$ has order $48$, it follows that $G$ has order 6. We have $c\cdot V_{1,1} = V_{1,1}$ and \[dbd\cdot V_{1,1} = db\cdot V_{-1,1} = d\cdot V_{-1,1} = V_{1,1},\]
so $\gen{c, dbd} \subset \Stab_\Gamma(V)$. Moreover, the map 
\[cdbd : (\lambda,\mu)\mapsto (-\mu^{-1}, \mu^{-1}\lambda)\]
has order 3 and $(cdbd)^2c = c(cdbd)$. Altogether, this shows that \[\gen{c,dbd} = \Stab_\Gamma(V) \isom S_3.\]

\textit{c) and d).} If $p$ is a point on $V$ with nontrivial stabilizer in $\Gamma$, then there exists $\gamma\in \Gamma\setminus \{\id\}$ with $p=\gamma p \in V\cap \gamma\cdot V$. So either $\Stab_G(p)$ is nontrivial or $p$ is in $V\cap V'$ with $V'\in M\setminus \{V\}$. First we look for fixed points of elements of $G$. A computation shows that only the elements of order 3 of $G$ have fixed points, namely $q_1$ and $q_2$ are both fixed by $cdbd$: if $(\lambda,\mu) = (-\mu^{-1},\lambda\mu^{-1})$, then $\mu = \lambda\mu^{-1} = -\mu^{-2}$, so $\mu^3 = -1$, and $(\lambda,\mu) \in \{q_1,q_2\}$.

Next, we determine the intersections between elements of $M$. Let $(\eps_1,\eps_2)$, $(\eps_3,\eps_4) \in \{\pm 1\}^2$. First, we show that $V_{\eps_1,\eps_2} \cap V_{\eps_3,\eps_4} = \emptyset$, if $(\eps_1,\eps_2) \neq (\eps_3,\eps_4)$. Note that this also implies $W_{\eps_1,\eps_2} \cap W_{\eps_3,\eps_4} = \emptyset$. If $(\lambda,\mu)\in V_{\eps_1,\eps_2}\cap V_{\eps_3,\eps_4}$, then by \eqref{Equation: Vs}, we have
\[\mu = \frac{\eps_1\eps_2\lambda}{\eps_1\lambda + 1} = \frac{\eps_3\eps_4\lambda}{\eps_3\lambda+ 1}.\]
It follows that
\[(\eps_1\eps_2\eps_3- \eps_1\eps_3\eps_4)\lambda + \eps_1\eps_2  - \eps_3\eps_4  = 0,\]
since $\lambda \neq 0$. If $\eps_2 = \eps_4$, then the coefficient of $\lambda$ is $0$ and $\eps_1\eps_2 = \eps_3\eps_4$, thus $\eps_1 = \eps_3$, which contradicts our assumption. Otherwise, $\eps_4 = -\eps_2$, and we have
\[2\eps_1\eps_2\eps_3\lambda = -\eps_1\eps_2 -\eps_2\eps_3.\]
So $\lambda \in \{0,1,-1\}$, and thus $(\lambda,\mu)\not\in P$.

Next, let $(\lambda,\mu)\in V_{\eps_1,\eps_2} \cap W_{\eps_3,\eps_4}$. Then we have $\mu = \eps_4+ \eps_3\eps_4\lambda$ by \eqref{Equation: Ws}, and together with Equation \eqref{Equation: Vs}, this yields
\begin{align*}
0 &= \eps_2(\eps_4 + \eps_3\eps_4\lambda)(\eps_1\lambda + 1) - \eps_1\lambda \\
  &= \eps_1\eps_2\eps_3\eps_4\lambda^2 + (\eps_1\eps_2\eps_4 + \eps_2\eps_3\eps_4 - \eps_1)\lambda + \eps_2\eps_4
\end{align*}
This equation has the solutions
\begin{align*}
\lambda_{1,2} &= \frac{\eps_1 - \eps_1\eps_2\eps_4 - \eps_2\eps_3\eps_4}{2\eps_1\eps_2\eps_3\eps_4} \pm \frac{\sqrt{3 - 2(\eps_1\eps_3 + \eps_2\eps_4 + \eps_1\eps_2\eps_3\eps_4)}}{2\eps_1\eps_2\eps_3\eps_4}
\end{align*}
Now, we can determine the intersection of $V_{1,1}$ with one of the other curves. One has
\begin{align*}
V_{1,1} \cap W_{1,1} &= V_{1,1} \cap a\cdot V_{1,1} &= \bigl\{q_1 = (e^{2i\pi /3}, e^{i\pi /3}), q_2 = (e^{-2i\pi /3}, e^{-i\pi /3})\bigr\}\\
V_{1,1} \cap W_{-1,1} &= V_{1,1} \cap da\cdot V_{1,1} &= \bigl\{r_1 = \Bigl(\tfrac{-1+\sqrt{5}}{2}, \tfrac{3-\sqrt{5}}{2}\Bigr), r_2 = \Bigl(\tfrac{-1-\sqrt{5}}{2}, \tfrac{3+\sqrt{5}}{2}\Bigr)\bigr\}\\
V_{1,1} \cap W_{-1,-1} &= V_{1,1} \cap ea\cdot V_{1,1} &= \bigl\{r_3 = \Bigl(\tfrac{1+\sqrt{5}}{2}, \tfrac{-1+\sqrt{5}}{2}\Bigr), r_4 = \Bigl(\tfrac{1-\sqrt{5}}{2}, \tfrac{-1-\sqrt{5}}{2}\Bigr)\bigr\}\\
V_{1,1} \cap W_{1,-1} &= V_{1,1} \cap eda\cdot V_{1,1} &= \bigl\{r_5 = \Bigl(\tfrac{-3+\sqrt{5}}{2}, \tfrac{1-\sqrt{5}}{2}\Bigr), r_6 = \Bigl(\tfrac{-3-\sqrt{5}}{2}, \tfrac{1+\sqrt{5}}{2}\Bigr)\bigr\}.
\end{align*}
The maps $a$ and $c$ both exchange $q_1$ and $q_2$, so $ca\in \Stab_\Gamma(q_i)$, $i=1,2$. Therefore, $\Stab_\Gamma(q_i) \supset \gen{ca,cdbd}$. Since $\Gamma\cdot q_i$ contains the 8-element set \[G\cdot q_1 \cup dG\cdot q_1 \cup eG\cdot q_1 \cup edG\cdot q_1,\]
$\Stab_\Gamma(q_i)$ is a group of order at most 6. But then $\Stab_\Gamma(q_i) = \gen{ca,cdbd}$ and it is again isomorphic to $S_3$ (observe that $(cdbd)^2ca = ca(cdbd)$ and $(ca)^2 = (cdbd)^3 = \id$).

The set $\{r_1,\dots,r_6\}$ is a $G$-orbit, for one has $cdbd\cdot r_1 = r_6$ and $(cdbd)^2\cdot r_1 = r_4$ and moreover $dbd\cdot r_1 = r_5$, $dbd\cdot r_6 = r_2$ and $dbd\cdot r_4 = r_3$. Furthermore, the sets
\[G\cdot r_1,\quad dG\cdot r_1, \quad eG\cdot r_1,\quad edG\cdot r_1\]
are all mutually disjoint 6-element sets, so there are at least 24 elements in $\Gamma\cdot r_1$. Thus it suffices to show that $\Stab_\Gamma(r_1)\neq 1$. But $da\cdot r_1 = r_2$ and $G$ acts transitively on $\{r_1,\dots,r_6\}$, so we find $g\in G$ with $g\cdot r_2 = r_1$ and $gda$ fixes $r_1$.
\end{proof}

\begin{proof}[Proof of Proposition \ref{Section 4 - Proposition: V/G to M_2 is injective}]
Note that the restriction of the canonical projection $\kappa: P\ra P/\Gamma$, $p\mapsto \Gamma\cdot p$ to $V$ factors through $V/G$. Moreover, $\pr = \barpr\circ \kappa$ (see Proposition \ref{Section 2 - Proposition: Surface S in the moduli space}). We proceed in two steps.

First, we show that $V/G \ra \kappa(V) \subset P/\Gamma$, $G\cdot x \mapsto \Gamma\cdot x$ is injective. Let $x,y\in V$ with $\Gamma\cdot x = \Gamma\cdot y$. Then $x \in V\cap \gamma\cdot V$ for some $\gamma\in \Gamma$. The proof of Lemma \ref{Section 4 - Lemma: Stabilizer of V} shows that $x \in G\cdot q_1$ or $x\in G\cdot r_1$ and the same holds for $y$. But $\Gamma\cdot q_1 \neq \Gamma\cdot r_1$, thus $G\cdot x = G\cdot y$.

Next, we consider the restriction of $\barpr:P/\Gamma\ra M_2$ to $\kappa(V)$. By Proposition \ref{Section 2 - Proposition: Where P mod Gamma to M_2 is not injective}, we know that $\barpr|_{\kappa(V)}$ is injective outside $\{\Gamma\cdot q_1,\Gamma\cdot r_1\}$, since $\Gamma\cdot q_1 \cup \Gamma\cdot r_1$ are precisely the points on $V$ with nontrivial stabilizer in $\Gamma$. Again by Proposition \ref{Section 2 - Proposition: Where P mod Gamma to M_2 is not injective}, we also know that $\pr(q_1) = Q'$ and $\pr(r_1) \neq Q'$, so altogether, $\barpr|_{\kappa(V)}$ is injective.
\end{proof}

\subsection{Veech group  and cusps}
In this section, we discuss the Veech group $\Gamma(S)$ of the origami curve $\Curve_S$, and use this description to determine its number of cusps, \ie points on the boundary of the moduli space. Recall that Proposition \ref{Section 1 - Proposition: Teichmueller Curve iff Veech group is lattice} states that $\Curve_S$ is birational to $\H/\hat{\Gamma}(S)$. But since $\Curve_S$ is regular, they are even isomorphic.

We will be working with $\H/\Gamma(S)$, which is anti-holomorphic to $\H/\hat\Gamma(S)$. Recall that Remark \ref{Section 1 - Remark: Primitive Origamis have entire Veech group} implies that the group $\Gamma(S)$ is a subgroup of finite index of $\SL_2(\Z)$. With the help of the algorithm in \cite{Sc04}, we can determine generators and coset representatives for $\Gamma(S)$. Let $s = \textmatrix{0}{-1}{1}{0}$ and $t = \textmatrix{1}{1}{0}{1}$ be the standard generators of $\SL_2(\Z)$. Then $\Gamma(S)$ is generated by
\[s^2,\ tst^{-2},\ sts^{-1},\ t^3,\ \txt{and}\ t^2st^{-1}.\]
Coset representatives for $\SL_2(\Z)/\Gamma(S)$ are given by
\[I,\ s,\ t,\ \txt{and}\ t^2.\]
Let $\Fund$ denote the hyperbolic pseudo-triangle with vertices $-\tfrac{1}{2}+i\tfrac{\sqrt{3}}{2}$, $\tfrac{1}{2}+i\tfrac{\sqrt{3}}{2}$ and $i\infty$, which is a fundamental domain for $\SL_2(\Z)$. Then a fundamental domain for the action of $\Gamma(S)$ on $\H$ is given by
\[\Fund_S = \Fund\,\cup\, t(\Fund)\, \cup\, s(\Fund)\, \cup\, t^2(\Fund).\]
A picture of $\Fund_S$ is given in Figure \ref{Figure: Fundamental domain}.
\begin{figure}[th]
\includegraphics[scale=0.3]{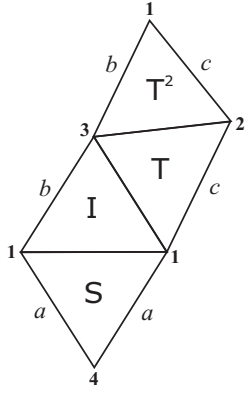}
\caption{A fundamental domain for $\Gamma(S)$ \label{Figure: Fundamental domain}}
\end{figure}
Here, edges that are labeled with the same letter are identified by the action of $\Gamma(S)$. Moreover, $\Fund_S$ is triangulated by 4 triangles with 6 edges and 4 vertices. Using Euler's formula we can thus reprove that $\Curve_S$ is of genus 0. Furthermore, $\H/\Gamma(S)$ has two cusps, namely the vertices 3 and 4. Therefore, $\Curve_S$ also has at most two cusps, and we shall show that it has precisely two.

Recall that $M_g$ is compactified by adding all stable Riemann surfaces of genus $g$. These are obtained by contracting a system of simple closed paths to points. In the case of origamis, it suffices to contract the system of core curves of certain cylinder decompositions of the surface in order to obtain all points on the boundary of the origami curve; see \eg \cite{HS06} for more details. 

The orbits of the two cusps of $\H/\Gamma(S)$ are represented by $0$ and $\infty \in \Fund_S$. To see whether they are distinct in $\bar{M_2}$, we travel along a path to the cusp and see what happens to the Riemann surfaces that correspond to points on that path.

The horizontal, respectively vertical saddle connections induce a decomposition of the origami into two cylinders $(C^h_i)_{i=1,2}$, respectively three cylinders $(C^v_j)_{j=1,2,3}$. Let $c^h_i$, respectively $c^v_j$ be the core curve of the cylinder $C^h_i$, respectively $C^v_j$.

\begin{figure}[ht]
\begin{center}
\setlength{\unitlength}{1cm}
\begin{picture}(11,2.5)
\thinlines
\put(0,0){\framebox(3,1)}
\put(1,1){\framebox(3,1)}
\thicklines
\drawline(0,0)(3,0)
\drawline(0,1)(3,1)
\drawline(1,1)(4,1)
\drawline(1,2)(4,2)
\thinlines

\dottedline[.]{0.1}(0,0.5)(3,0.5)
\dottedline[.]{0.1}(1,1.5)(4,1.5)

\put(-.1,-0.1){\framebox(.2,.2)}
\put(-.1,0.9){\framebox(.2,.2)}
\put(2.9,-0.1){\framebox(.2,.2)}
\put(2.9,0.9){\framebox(.2,.2)}
\put(2.9,1.9){\framebox(.2,.2)}

\put(.9,-0.1){$\blacksquare$}
\put(.9,0.9){$\blacksquare$}
\put(.9,1.9){$\blacksquare$}
\put(3.9,0.9){$\blacksquare$}
\put(3.9,1.9){$\blacksquare$}

\put(7,0){\framebox(1,1)}
\put(8,0){\framebox(2,2)}
\put(10,1){\framebox(1,1)}

\thicklines
\drawline(7,0)(7,1)
\drawline(8,0)(8,2)
\drawline(10,0)(10,2)
\drawline(11,1)(11,2)
\thinlines

\dottedline[.]{0.1}(7.5,0)(7.5,1)
\dottedline[.]{0.1}(9,0)(9,2)
\dottedline[.]{0.1}(10.5,1)(10.5,2)

\put(6.9,-0.1){\framebox(.2,.2)}
\put(6.9,0.9){\framebox(.2,.2)}
\put(9.9,-0.1){\framebox(.2,.2)}
\put(9.9,0.9){\framebox(.2,.2)}
\put(9.9,1.9){\framebox(.2,.2)}

\put(7.9,-0.1){$\blacksquare$}
\put(7.9,0.9){$\blacksquare$}
\put(7.9,1.9){$\blacksquare$}
\put(10.9,0.9){$\blacksquare$}
\put(10.9,1.9){$\blacksquare$}
\end{picture}
\caption{Horizontal and vertical cylinder decomposition of $S$}
\end{center}
\end{figure}
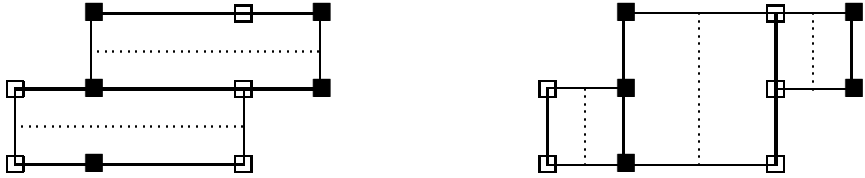

Traveling along the path $\textmatrix{e^{t}}{0}{0}{e^{-t}}\cdot i \in \Fund_S$ ($t\in [0,\infty)$) towards $\infty$ corresponds to pinching the core curves $(c^h_i)$ of the horizontal cylinders. In the limit, we obtain a stable curve $X^h$ with one irreducible component of genus 0 and two nodes. Likewise, traveling along the path $\textmatrix{e^{-t}}{0}{0}{e^{t}}\cdot i \in \Fund_S$ ($t\in [0,\infty)$) towards 0 corresponds to pinching the core curves $(c^v_j)$ of the vertical cylinders. In the limit, we obtain a stable curve $X^v$, which is different from the first, since it consists of two irreducible components of genus 0, which intersect in 3 nodes. The dual graphs to the stable curves $X^h$ and $X^v$ are depicted in Figure \ref{Figure: dual graphs}.

\begin{figure}[ht]
\begin{center}
\setlength{\unitlength}{0.7cm}
\begin{picture}(4,2)
\put(1,1){\bigcircle{2}}
\put(3,1){\bigcircle{2}}
\put(2,1){\circle*{0.2}}
\put(1.6,1){$0$}
\end{picture}
\hspace{1.5cm}
\begin{picture}(4,2)
\put(0,1){\curve(0,0, 2,1, 4,0)}
\put(0,1){\curve(0,0, 2,-1, 4,0)}
\put(0,1){\line(1,0){4}}
\put(0,1){\circle*{0.2}}
\put(4,1){\circle*{0.2}}
\put(-0.4,1){$0$}
\put(4.2,1){$0$}
\end{picture}
\caption{Dual graphs of the stable curves $X^h$ and $X^v$ \label{Figure: dual graphs}}
\end{center}
\end{figure}
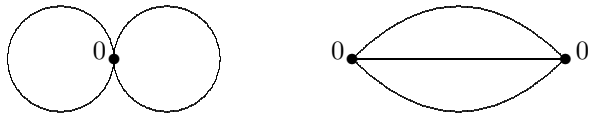

We sum up these observations in the following proposition.
\begin{prop}
The closure of the origami curve $\Curve_S$ in $\bar{M_2}$ is isomorphic to the projective line. It intersects the boundary $\partial\bar{M_2}$ in two distinct points.
\end{prop}

\bibliographystyle{amsalpha}
\bibliography{kappes_biblio}
\end{document}